\documentclass[a4paper]{amsart}
\usepackage[utf8]{inputenc}
\usepackage{mathtools}
\usepackage{amssymb,amsfonts,amsmath}
\usepackage{enumitem}
\usepackage{color}
\usepackage{mathrsfs}
\usepackage{graphicx}
\usepackage{verbatim}
\usepackage{tikz-cd}
\usepackage{tikz}
\usetikzlibrary{arrows.meta,positioning,calc,decorations.pathmorphing}
\usepackage{url}
\usepackage{mathabx}
\usepackage[dvipsnames]{xcolor}
\usepackage{aliascnt}
\usepackage[capitalise,nameinlink,noabbrev]{cleveref}
\numberwithin{equation}{section}
\newtheorem{theorem}{Theorem}[section]
\newaliascnt{corollary}{theorem}
\newtheorem{corollary}[corollary]{Corollary}
\aliascntresetthe{corollary}
\newaliascnt{lemma}{theorem}
\newtheorem{lemma}[lemma]{Lemma}
\aliascntresetthe{lemma}
\newaliascnt{goal}{theorem}

\aliascntresetthe{goal}
\newaliascnt{proposition}{theorem}
\newtheorem{proposition}[proposition]{Proposition}
\aliascntresetthe{proposition}
\newtheorem*{question}{Question}
\newaliascnt{problem}{theorem}

\aliascntresetthe{problem}
\newaliascnt{assumption}{theorem}

\aliascntresetthe{assumption}
\newaliascnt{conjecture}{theorem}

\aliascntresetthe{conjecture}
\newtheorem{introtheorem}{Theorem}

\theoremstyle{definition}

\theoremstyle{definition}
\newaliascnt{definition}{theorem}
\newtheorem{definition}[definition]{Definition}
\aliascntresetthe{definition}
\newaliascnt{notation}{theorem}
\newtheorem{notation}[notation]{Notation}
\aliascntresetthe{notation}
\newaliascnt{convention}{theorem}
\newtheorem{convention}[convention]{Convention}
\aliascntresetthe{convention}
\newaliascnt{construction}{theorem}

\aliascntresetthe{construction}
\theoremstyle{remark}
\newaliascnt{remark}{theorem}
\newtheorem{remark}[remark]{Remark}
\aliascntresetthe{remark}
\newaliascnt{remarks}{theorem}

\aliascntresetthe{remarks}
\newaliascnt{example}{theorem}

\aliascntresetthe{example}
\crefname{theorem}{Theorem}{Theorems}
\crefname{corollary}{Corollary}{Corollaries}
\crefname{lemma}{Lemma}{Lemmas}
\crefname{goal}{Goal}{Goals}
\crefname{proposition}{Proposition}{Propositions}
\crefname{question}{Question}{Questions}
\crefname{problem}{Problem}{Problems}
\crefname{assumption}{Assumption}{Assumptions}
\crefname{conjecture}{Conjecture}{Conjectures}
\crefname{definition}{Definition}{Definitions}
\crefname{notation}{Notation}{Notations}
\crefname{convention}{Convention}{Conventions}
\crefname{construction}{Construction}{Constructions}
\crefname{remark}{Remark}{Remarks}
\crefname{remarks}{Remarks}{Remarks}
\crefname{example}{Example}{Examples}
\Crefname{theorem}{Theorem}{Theorems}
\Crefname{corollary}{Corollary}{Corollaries}
\Crefname{lemma}{Lemma}{Lemmas}
\Crefname{goal}{Goal}{Goals}
\Crefname{proposition}{Proposition}{Propositions}
\Crefname{question}{Question}{Questions}
\Crefname{problem}{Problem}{Problems}
\Crefname{assumption}{Assumption}{Assumptions}
\Crefname{conjecture}{Conjecture}{Conjectures}
\Crefname{definition}{Definition}{Definitions}
\Crefname{notation}{Notation}{Notations}
\Crefname{convention}{Convention}{Conventions}
\Crefname{construction}{Construction}{Constructions}
\Crefname{remark}{Remark}{Remarks}
\Crefname{remarks}{Remarks}{Remarks}
\Crefname{example}{Example}{Examples}

\newcommand{\N}{\mathbb{N}}

\newcommand{\Z}{\mathbb{Z}}

\newcommand{\CC}{\mathfrak{C}}
\newcommand{\OO}{\mathcal{O}}

\DeclareMathOperator{\Aut}{Aut}

\DeclareMathOperator{\CAT}{CAT}
\DeclareMathOperator{\Cay}{Cay}

\DeclareMathOperator{\EL}{EL}
\DeclareMathOperator{\F}{F}

\DeclareMathOperator{\GL}{GL}

\DeclareMathOperator{\HNN}{HNN}

\DeclareMathOperator{\Homeo}{Homeo}

\DeclareMathOperator{\pr}{pr}

\DeclareMathOperator{\RiSt}{RiSt}

\DeclareMathOperator{\St}{St}
\DeclareMathOperator{\supp}{supp}
\DeclareMathOperator{\Sym}{Sym}
\DeclareMathOperator{\T}{T}

\newcommand{\abs}[1]{\vert #1 \vert}
\newcommand\Set[2]{\{\,#1\mid#2\,\}}

\newcommand{\defeq}{\mathrel{\mathop{:}}=}

\renewcommand{\epsilon}{\varepsilon}

\title[Non-uniform exponential growth]{Non-uniform exponential growth and the decay of growth rates in growing dimensions}
\author[R. Sauer]{Roman Sauer}
\address{Karlsruhe Institute of Technology, Englerstr.\ 2, 76131 Karlsruhe, Germany}
\email{roman.sauer@kit.edu}
\author[E. Schesler]{Eduard Schesler}
\address{Karlsruhe Institute of Technology, Englerstr.\ 2, 76131 Karlsruhe, Germany}
\email{eduardschesler@googlemail.com}
\begin{document}
\begin{abstract}
We provide the first example of a finitely presented, and the first example of a simple, group of non-uniform exponential growth.
The example is given by Thompson's group $V$.
Our methods also show that the infimal exponential growth rates of $\Aut(F_{2^{n+2}})$ and of $\EL_{2^{n+2}}(R)$, for every finitely generated ring $R$, tend to $1$.
As an application, we obtain the first example of an acylindrically hyperbolic group, and the first example of a Kazhdan group, of non-uniform exponential growth.
\end{abstract}
\maketitle

\section{Introduction}

One of the most natural invariants that can be attached to a finitely generated group $G$ with finite generating set $S$ is the \emph{(word) growth function} $\gamma_G^S$, which assigns to each $n \in \mathbb{N}$ the number $\gamma_G^S(n)$ of elements in $G$ that can be represented by
words of length at most $n$ over $S \cup S^{-1}$.
The systematic study of this invariant dates back to the foundational works of Efremovich~\cite{Efremovich1953}, {\v S}varc~\cite{Svarc1955} and Milnor~\cite{Milnor68} in the 1950s and 1960s, and has remained a central theme of group theory ever since.
For a recent historical survey on the word growth of groups we refer to~\cite{delaHarpe25}.
A celebrated result in this area is Gromov's characterization of finitely generated groups of polynomial growth as virtually nilpotent groups~\cite{Gromov1981}.
In the special case of solvable groups, this characterization had been
established earlier by Milnor~\cite{Milnor1968b} and Wolf~\cite{Wolf1968},
who showed that a finitely generated solvable group has polynomial growth
if it is virtually nilpotent, and exponential growth otherwise.
In particular, no finitely generated solvable group has growth strictly between polynomial and exponential.
This led Milnor~\cite{Milnor68c} to raise the question of whether every finitely generated group has either polynomial or exponential growth.
This question was answered in the negative by Grigorchuk~\cite{Grigorchuk84}, who introduced a group $\mathcal{G}$, now known as the first Grigorchuk group, whose growth function $\gamma_{\mathcal{G}}^S$ satisfies
\[
  e^{\sqrt{n}} \preceq \gamma_{\mathcal{G}}^S(n) \preceq e^{n^{\alpha}}
  \qquad\text{for some } 1/2 < \alpha < 1,
\]
where $\preceq$ refers to the standard preorder on growth functions, given by $f \preceq g$ if there exists a constant $C > 0$ such that $f(n) \leq C \cdot g(Cn)$ for all $n \in \mathbb{N}$.
The optimal exponent in the upper bound was later determined by Bartholdi~\cite{Bartholdi98}.
Setting
\[
\alpha_0 \defeq \log(2)/\log(2/\eta) \approx 0.7674,
\]
where $\eta$ denotes the unique real root of the polynomial $x^3+x^2+x-2$, he showed that $\gamma_{\mathcal{G}}^S(n) \preceq e^{n^{\alpha_0}}$, while a more recent result of Erschler and Zheng~\cite{ErschlerZheng20} shows that $e^{n^{\beta}} \preceq \gamma_{\mathcal{G}}^S(n)$ for every $\beta < \alpha_0$.
The constant $\alpha_0$ will appear throughout our quantitative results.
To this day it remains a major open problem in group theory, known as Grigorchuk's \emph{gap conjecture}~\cite{Grigorchuk2014}, whether every finitely generated group whose growth function is strictly slower than $e^{\sqrt{n}}$ is virtually nilpotent.
To every growth function $\gamma_G^S$ one can associate the limit
\[
\omega(G,S) = \lim_{n\to\infty} \sqrt[n]{\gamma_G^S(n)},
\]
called the \emph{exponential growth rate} of $G$ with respect to $S$, whose existence was established by Milnor~\cite{Milnor68c}.
Note that $G$ has \emph{exponential growth}, that is, $\gamma_G^S$ is bounded below by $C^n$ for some constant $C > 1$, if and only if $\omega(G,S) > 1$ for one, equivalently every, finite generating set $S$.
As the notation suggests, in this case the rate $\omega(G,S)$ depends on $S$.
Indeed, by enlarging the generating set, $\omega(G,S)$ can be made arbitrarily large.
However, it is far more subtle to control how small the rates $\omega(G,S)$ can become as the generating set varies.
This motivates the study of the infimum
\[
\omega(G) := \inf_S \omega(G,S),
\]
taken over all finite generating sets of $G$.
This quantity, the \emph{infimal exponential growth rate} of $G$, is a genuine invariant of the group.
One says that $G$ has \emph{uniform exponential growth} if $\omega(G) > 1$.
Otherwise, if $G$ has exponential growth but $\omega(G) = 1$, we say that $G$ has \emph{non-uniform exponential growth}.
The distinction between exponential growth and uniform exponential growth goes back to a question of Gromov~\cite{Gromov1981}, who asked whether every finitely generated group of exponential growth has uniform exponential growth.
This question became highly influential and motivated a large body of work in geometric group theory.
For many natural classes of groups it has been shown that Gromov's question has an affirmative answer.
These include linear groups~\cite{EskinMozesOh2005,BreuillardGelander2008}, elementary amenable groups~\cite{Osin03,Breuillard07}, and an expanding list of groups that act on hyperbolic spaces~\cite{Koubi1998,Xie07,AbbottNgSpriano24,LegaspiSteenbock2026} and low-dimensional $\CAT(0)$ cube complexes~\cite{BucherdelaHarpe00,KarSageev19,GuptaJankiewiczNg23}.
Thus, historically, uniform exponential growth appeared to be a robust feature among the known groups of exponential growth.
However, Gromov's question was ultimately shown to have a negative answer by
Wilson~\cite{Wilson04a}, whose examples come, like the first Grigorchuk group, from the class of branch groups.
As with all known examples of branch groups, Wilson's groups are not finitely presented.
This gives rise to the natural question, first raised by Mann~\cite{Mann11}, of whether finite presentability is compatible with non-uniform exponential growth.

\begin{question}
Does there exist a finitely presented group of non-uniform exponential growth?
\end{question}

An interesting perspective on groups of non-uniform exponential growth is offered by the class of just-infinite groups.
Recall that a group is \emph{just-infinite} if it is infinite and every proper quotient is finite.
By a standard application of Zorn's lemma, every finitely generated infinite group admits a just-infinite quotient.
Furthermore, the infimal exponential growth rate does not increase under taking quotients.
Consequently, if $G$ is of non-uniform exponential growth, then every just-infinite quotient of $G$ is either of non-uniform exponential growth or of subexponential growth.
This suggests that the class of just-infinite groups has a higher density of groups of non-uniform exponential growth than the class of all finitely generated groups, and hence provides a suitable class in which to search for them.
By results of Wilson~\cite{Wilson71} and Grigorchuk~\cite{Grigorchuk00}, the study of finitely generated just-infinite groups reduces to the study of three subclasses: branch groups, hereditarily just-infinite groups, and simple groups.
This raises the question of which of these subclasses contain groups of non-uniform exponential growth.
Wilson's original construction~\cite{Wilson04a} of groups of non-uniform exponential growth provides examples in the class of branch groups, which are in particular residually finite.
It therefore remains to answer the following questions. The first one appears, for example, in Bartholdi's lecture notes~\cite{bartholdi-lectures} or in the Kourovka Notebook~\cite[21.117]{KourovkaNotebook26}.

\begin{question}
Does there exist a finitely generated simple group of non-uniform exponential growth?
\end{question}

\begin{question}
Does there exist a finitely generated hereditarily just-infinite group of non-uniform exponential growth?
\end{question}

The main result of this paper provides an affirmative answer to the first two questions above.
To our great surprise, the group we consider for this purpose does not need to be constructed.
It is Thompson's group $V$, one of the classical finitely presented infinite simple groups, introduced by Richard Thompson in the 1960s, long before Gromov raised his question.
In particular this shows that non-uniform exponential growth arises naturally.

\begin{introtheorem}\label{thm:main}
Thompson's group $V$ is of non-uniform exponential growth.
\end{introtheorem}

This result appears even more surprising when contrasted with the behaviour of Thompson's group $F$, which is closely related to $V$ but has uniform exponential growth, as shown by de la Harpe~\cite{delaHarpe2002}.
\Cref{thm:main} also contrasts with the aforementioned results on groups acting on low-dimensional $\CAT(0)$ cube complexes~\cite{BucherdelaHarpe00,KarSageev19,GuptaJankiewiczNg23}. Indeed, Farley showed that Thompson's group $V$ acts properly by isometries on an infinite-dimensional $\CAT(0)$ cube complex~\cite{Farley05}.
Thus the results concerning low dimensions do not extend to proper actions on infinite-dimensional $\CAT(0)$ cube complexes.
Although the techniques in our proof of \cref{thm:main} are flexible enough to apply to several relatives of Thompson's group $V$, they break down for Thompson's group $T$.
Since $T$ is simple, the argument by which de la Harpe~\cite{delaHarpe2002} established uniform exponential growth for Thompson's group $F$, namely realizing it as a non-ascending $\HNN$-extension of itself, is not available either.
This leaves us with the following question.

\begin{question}
Is Thompson's group $T$ of uniform exponential growth?
\end{question}

We now turn to groups acting on hyperbolic spaces, a setting in which the study of Gromov's question has a long history.
By results of Balasubramanya, Fournier-Facio and Genevois~\cite{BalasubramanyaFournierFacioGenevois25}, this class is essentially disjoint from the world of Thompson's groups in the sense that no isometric action of $V$, or of many of its relatives, on a hyperbolic space admits a loxodromic element.
The rich supply of affirmative results concerning Gromov's question in this setting started with the work of Koubi~\cite{Koubi1998}, who established uniform exponential growth for non-elementary hyperbolic groups.
Xie~\cite{Xie07} subsequently proved that relatively hyperbolic groups are either virtually cyclic or of uniform exponential growth.
For the even more general class of acylindrically hyperbolic groups, however, the corresponding question was already noted as open in the early days of this class~\cite{MinasyanOsin2018} and continues to attract interest~\cite{AbbottNgSpriano24,LegaspiSteenbock2026}.

\begin{question}
Does every acylindrically hyperbolic group of exponential growth have uniform exponential growth?
\end{question}

Prominent examples of acylindrically hyperbolic groups include mapping class groups~\cite{Bowditch08}, for which uniform exponential growth was established by Anderson, Aramayona and Shackleton~\cite{AndersonAramayonaShackleton07} and extended by Mangahas~\cite{Mangahas10} to all of their subgroups of exponential growth.
Another interesting family of acylindrically hyperbolic groups, which will play an important role in this article, is given by the automorphism groups $\Aut(F_n)$ of finitely generated free groups~\cite{GenevoisHorbez2020}, for which uniform exponential growth follows from the corresponding property for their quotients $\GL_n(\Z)$.
Recently, Abbott, Ng and Spriano~\cite{AbbottNgSpriano24} established uniform exponential growth for every virtually torsion-free acylindrically hyperbolic group that is hierarchically hyperbolic, thereby recovering several previously known results in a unified way.
Beyond its intrinsic interest, the question about the uniform exponential growth of acylindrically hyperbolic groups is closely related to the question of whether the infimal exponential growth rates of all non-elementary hyperbolic groups admit a common lower bound $c>1$, see~\cite[Question~2.1]{Osin04}.
Remarkably, each of these infimal rates is in fact a minimum, as was recently shown by Fujiwara and Sela~\cite{FujiwaraSela23}.
By a result of Minasyan and Osin~\cite{MinasyanOsin2018} there is a finitely generated acylindrically hyperbolic group that is a quotient of every non-elementary hyperbolic group.
Since exponential growth rates do not increase under passing to quotients, an affirmative answer to the question about the uniform exponential growth of acylindrically hyperbolic groups would imply the existence of a lower bound $c$ as above.
This strategy was suggested by Osin~\cite{Osin2002} to construct groups of non-uniform exponential growth even before Wilson's result and has continued to attract attention, see e.g.~\cite{LegaspiSteenbock2026}.
The following result rules out this strategy.
Moreover, it shows that non-uniform exponential growth can be combined with Kazhdan's property~$(\T)$, which answers a question of de la Harpe~\cite{delaHarpe2002}.

\begin{introtheorem}\label{thm:acylhyp}
There is an acylindrically hyperbolic group of non-uniform exponential growth with Kazhdan's property~$(\T)$.
\end{introtheorem}

Since acylindrically hyperbolic groups are SQ-universal~\cite{DahmaniGuirardelOsin17}, we obtain from \cref{thm:acylhyp} that \emph{every countable group embeds into a group of non-uniform exponential growth with Kazhdan's property~$(\T)$}. This generalizes a result of Bartholdi and Erschler~\cite{BartholdiErschler15}, who proved the corresponding statement without property~$(\T)$.
Earlier work already indicated that acylindrically hyperbolic groups can exhibit substantially more exotic behaviour than relatively hyperbolic groups.
This is illustrated by the failure of the standard strategy for proving uniform exponential growth, which consists in establishing a uniform Tits alternative for monoids by finding a constant $R$ such that, with respect to every finite generating set, the ball of radius $R$ contains two elements generating a free submonoid.
Minasyan and Osin~\cite{MinasyanOsin2018} exhibited this failure by means of a common-quotient construction, producing an acylindrically hyperbolic group with arbitrarily large torsion balls with respect to suitable finite generating sets.
More recently, Legaspi and Steenbock~\cite{LegaspiSteenbock2026} recovered this phenomenon using small-cancellation methods which they showed to preserve uniform exponential growth.
One perspective offered by the common quotient theorem of Minasyan and Osin~\cite{MinasyanOsin2018} is that it reduces the construction of an acylindrically hyperbolic group of non-uniform exponential growth to showing that the infimum of the infimal exponential growth rates of all acylindrically hyperbolic groups is equal to~$1$.
Our proof of \cref{thm:acylhyp} follows this approach and is realized via the following result.

\begin{introtheorem}\label{thm:mainestimate}
There are finite generating sets $\mathcal{A}_n$ of $\Aut(F_{2^{n+2}})$ such that for every $\alpha\in(\alpha_0,1)$ there is a constant $C>0$ with
\[
\gamma^{\mathcal{A}_n}_{\Aut(F_{2^{n+2}})}(\ell)\le \exp\bigl(C \ell^{\alpha}\bigr)
\quad\text{for every}\quad 0\le \ell\le \tfrac{2^{n}-2}{3}.
\]
As a consequence, for every $\beta \in (0,1-\alpha_0)$ we have
\[
  \omega\bigl(\Aut(F_{2^{n+2}})\bigr) \leq \exp\bigl(2^{-\beta(n+2)}\bigr)
\]
for all sufficiently large $n$.
\end{introtheorem}

Since $\GL_n(\Z)$ is a quotient of $\Aut(F_n)$, \cref{thm:mainestimate} immediately yields the analogous estimate for the groups $\GL_{2^{n+2}}(\Z)$.
In particular, the infimal exponential growth rates $\omega(\GL_{2^{n+2}}(\Z))$ tend to $1$ as $n \to \infty$.
In contrast, Breuillard's uniform Tits alternative~\cite{Breuillard11} provides for every $d$ a constant $c(d)>1$ such that $\omega(\Gamma)\geq c(d)$ for every field $K$ and every subgroup $\Gamma \leq \GL_d(K)$ that is not virtually solvable.
The assertion that this conclusion extends to all subgroups of exponential growth is known as Breuillard's growth conjecture~\cite{Breuillard14} and was shown by Breuillard and Varj\'u to be equivalent to Lehmer's conjecture on Mahler measures~\cite{BreuillardVarju20}.
By a construction of Grigorchuk and de la
Harpe~\cite{GrigorchukdelaHarpe01lim} it is known that the optimal such constants satisfy $\inf_{d\in\N} c(d) = 1$.
They exhibited a sequence of groups, defined by truncated presentations of the Grigorchuk group, which are commensurable with direct products of non-abelian free groups.
In particular, these groups embed into $\GL_d(\Z)$ for suitable $d$ and are not virtually solvable, while their growth rates with respect to fixed generating sets tend to~$1$, see also~\cite[Remark~1.4]{Breuillard08}.
\Cref{thm:mainestimate} implies that this decay of infimal growth rates is witnessed by the groups $\GL_d(\Z)$ themselves.
In fact, we prove the following more general result, in which the ring $\Z$ is replaced by an arbitrary finitely generated associative (unital) ring $R$.
Since the groups $\GL_n(R)$ need not be finitely generated in this generality, we formulate it for the subgroup $\EL_n(R) \leq \GL_n(R)$ generated by elementary matrices, which is finitely generated for every $n \geq 3$ whenever $R$ is finitely generated.

\begin{introtheorem}\label{thm:el}
For every finitely generated associative unital ring $R$ and every $\beta \in (0,1-\alpha_0)$ we have
\[
\omega\bigl(\EL_{2^{n+2}}(R)\bigr) \leq \exp\bigl(2^{-\beta(n+2)}\bigr)
\]
for all sufficiently large $n$.
In particular, the infimal exponential growth rates
$\omega(\EL_{2^{n+2}}(R))$ tend to $1$ as $n\to\infty$.
\end{introtheorem}

The proofs of Theorems~\ref{thm:main}--\ref{thm:el} have a common architecture and share several ingredients. 
An important one is the work of Bartholdi and Erschler on
inverted orbits of Grigorchuk's group~\cite{BartholdiErschler12}.
Another, more implicit, ingredient is that our techniques for proving small infimal exponential growth rates require high virtual cohomological dimension.
As a consequence, all groups that we prove to be of non-uniform exponential growth have infinite virtual cohomological dimension.
Whether this is an artefact of our methods or a genuine necessity is the content of the following question.

\begin{question}
Does there exist a finitely generated group of non-uniform exponential growth that has finite virtual cohomological dimension?
\end{question}

\subsection*{Acknowledgments}
The authors are grateful to Sean Eberhard, Matteo Migliorini and Volodymyr Nekrashevych for a number of helpful discussions. 

\subsection*{Declaration on the use of AI}
The results of this paper were proved in April/May 2026. The current paper merges and extends two preprints by the authors from May and June 2026. 

No form of AI was used for generating the ideas or proofs of this work. 
We did use ChatGPT 5.5 for editorial assistance and help with the TikZ figures and ChatGPT 5.6 for the final proofreading. 

\section{Background on Thompson's group $V$}\label{sec:V-background}

In this section we fix the notation for Thompson's group $V$ and collect some results that will be used throughout the paper.
Let $\mathfrak{C} \defeq \{0,1\}^{\N}$ denote the standard Cantor set, viewed as the set of infinite binary sequences.
We write $\{0,1\}^{\ast} \defeq \coprod_{\ell \in \N_0} \{0,1\}^{\ell}$ to denote the set of finite binary words.
The length of a word $w \in \{0,1\}^{\ast}$ will be denoted by $|w|$.
For every $w \in \{0,1\}^{\ast}$ we refer to
\begin{equation}\label{eq:cylinder}
w\mathfrak{C} \defeq \Set{w \xi}{\xi \in \mathfrak{C}}
\end{equation}
as the \emph{cylinder} associated to $w$, and more generally we write $A \mathfrak{C} \defeq \bigcup_{w \in A} w\mathfrak{C}$ for every subset $A \subseteq \{0,1\}^{\ast}$.
A finite subset $A \subseteq \{0,1\}^{\ast}$ is called a \emph{dyadic partition set of $\mathfrak{C}$} if
\[
\mathfrak{C} = \bigsqcup_{w \in A} w \mathfrak{C}.
\]
Equivalently, a dyadic partition set corresponds to the set of leaves of a finite
binary rooted tree.

\begin{definition}\label{def:V}
\emph{Thompson's group} $V$ is the subgroup of $\Homeo(\mathfrak{C})$ consisting of those homeomorphisms $\gamma$ for which there exist dyadic partition sets $A,B$ of $\mathfrak{C}$ of the same cardinality and a bijection $f \colon A \rightarrow B$ such that
\[
\gamma(w \xi) = f(w) \xi
\]
for every $w \in A$ and every $\xi \in \mathfrak{C}$.
\end{definition}

For background on $V$ and its relatives we refer to~\cite{CannonFloydParry1996,Hig74}.

\begin{notation}\label{not:support}
For every $\gamma \in V$ we write
\[
\supp(\gamma) \defeq \Set{\xi \in \mathfrak{C}}{\gamma(\xi) \neq \xi}
\]
to denote the \emph{support of $\gamma$}.
\end{notation}

\begin{notation}\label{not:local-V}
For every $\gamma \in V$ and every word $w \in \{0,1\}^{\ast}$, we write $w\gamma \in V$ to denote the unique element with $\supp(w\gamma) \subseteq w\mathfrak{C}$ that satisfies
\[
(w\gamma)(w\xi) = w \gamma(\xi)
\qquad \text{for every } \xi \in \mathfrak{C}.
\]
More generally, for every subgroup $H \leq V$, we write
\[
wH \defeq \Set{wh}{h \in H} \leq V.
\]
\end{notation}

The map $\gamma \mapsto w\gamma$ is an injective homomorphism $V \hookrightarrow V$ whose image consists of those elements that are supported on the cylinder $w\mathfrak{C}$.
In what follows we will work with a family of \emph{Higman--Thompson groups} $V_{Z}$, where $Z$ is a non-empty finite set.
To define them, we consider the disjoint union
\begin{equation}\label{eq:cantor-copies}
\mathfrak{C}_{Z} \defeq \bigsqcup_{z \in Z} \mathfrak{C}
\end{equation}
of copies of $\mathfrak{C}$ indexed by $Z$, and write $\mathfrak{C}_{z}$ for the $z$-th copy.
Extending our earlier terminology in the natural way, we call a finite subset $A \subseteq \bigsqcup_{z \in Z} \{0,1\}^{\ast}$ a \emph{dyadic partition set of $\mathfrak{C}_{Z}$} if for each $z \in Z$ the intersection of $A$ with the $z$-th summand $\{0,1\}^{\ast}$ is a dyadic partition set of $\mathfrak{C}_z$ in the sense above.
By definition, $V_{Z}$ consists of those homeomorphisms $\gamma$ of $\mathfrak{C}_{Z}$ for which there exist two dyadic partition sets $A_1$ and $A_2$ of $\mathfrak{C}_{Z}$ of the same cardinality and a bijection $f \colon A_1 \rightarrow A_2$ such that
\[
\gamma(w \xi) = f(w) \xi
\qquad \text{for every } w \in A_1 \text{ and every } \xi \in \mathfrak{C}.
\]
For $Z = \{1, \ldots, r\}$ we recover the standard definition of the Higman--Thompson group $V_{2,r} = V_{\{1,\ldots,r\}}$.
In this terminology, Thompson's group $V$ coincides with $V_{\{1\}}$.

It will be convenient to have an explicit set of generators of $V$.
Following Cannon--Floyd--Parry~\cite{CannonFloydParry1996}, we work with the standard generators $X_0,X_1,X_2,\ldots$ of the subgroup $\F \leq V$ consisting of orientation-preserving piecewise dyadic homeomorphisms.
Concretely, $X_0 \in \F$ is the homeomorphism given by
\begin{equation}\label{eq:thompson-X0}
X_0(00\xi) = 0\xi,
\quad
X_0(01\xi) = 10\xi,
\quad
X_0(1\xi) = 11\xi
\qquad \text{for every } \xi \in \mathfrak{C},
\end{equation}
and, for each $n \geq 1$, the element $X_n$ is given by $X_n = 1^n X_0$.
The group $\F$ is generated by $\{X_0, X_1\}$.
We refer to~\cite{CannonFloydParry1996} for a detailed account.

\begin{notation}\label{not:rigid-stabilizer}
Let $\mathfrak{X}$ be a topological space and let $G \leq \Homeo(\mathfrak{X})$ be a group of homeomorphisms of $\mathfrak{X}$.
For every subset $U \subseteq \mathfrak{X}$, the \emph{rigid stabilizer} of $U$ in $G$ is the subgroup
\[
\RiSt_G(U) \defeq \Set{\gamma \in G}{\gamma(\xi) = \xi \text{ for every } \xi \in \mathfrak{X}\, \setminus U}
\]
of those elements of $G$ that fix the complement of $U$ pointwise.
\end{notation}

The following lemma is an adaptation of~\cite[Lemma~4.4]{BleakElliottHyde24} to our terminology.

\begin{lemma}\label{lem:rigid-stabilizer-union}
Let $Z$ be a non-empty finite set and let $U_1, U_2 \subseteq \mathfrak{C}_{Z}$ be clopen subsets with $U_1 \cap U_2 \neq \emptyset$.
Then
\[
\langle \RiSt_{V_{Z}}(U_1), \RiSt_{V_{Z}}(U_2) \rangle = \RiSt_{V_{Z}}(U_1 \cup U_2).
\]
\end{lemma}

\section{Background on Grigorchuk's group $\mathcal{G}$}\label{sec:Grigorchuk-background}

In this section we fix some notation for the first Grigorchuk group, which we denote by $\mathcal{G}$, and collect the results about $\mathcal{G}$ that will be used in the remainder of the paper.
We refer to~\cite{BartholdiGrigorchukSunic2003} for a comprehensive introduction to $\mathcal{G}$ and to branch groups in general.
Let $\T$ denote the binary rooted tree with vertex set $\{0,1\}^{\ast}$ in which two vertices $v,w$ are connected by an edge if and only if either $v = wx$ or $w = vx$ for some $x \in \{0,1\}$.
For each $n \in \N_0$, the \emph{$n$-th level} of $\T$ is the set $L_n \defeq \{0,1\}^n$ of vertices at distance $n$ from the root.
The group of automorphisms of $\T$ will be denoted by $\Aut(\T)$.
The root of $\T$ is the only vertex of valence $2$ in $\T$ and is therefore fixed by $\Aut(\T)$.
Since moreover every element of $\Aut(\T)$ preserves the distance between vertices, it follows that $\Aut(\T)$ preserves each level $L_n$.
In particular, for every subgroup $G \leq \Aut(\T)$ we obtain a natural homomorphism from $G$ to the symmetric group $\Sym(L_n)$, which we denote by $\pi_n$.
We write $\sigma_g \in \Sym(\{0,1\})$ to denote the image of $g$ under $\pi_1$.
Conversely, we identify each permutation $\sigma \in \Sym(\{0,1\})$ with the automorphism of $\T$ given by $xw \mapsto \sigma(x)w$ for all $x \in \{0,1\}$ and $w \in \{0,1\}^{\ast}$.
The two automorphisms of $\T$ that arise in this way are called \emph{rooted}.
Another way to produce automorphisms of $\T$ is to act on the two subtrees below the root independently: Each pair $(g_0, g_1) \in \Aut(\T) \times \Aut(\T)$ gives rise to the automorphism of $\T$ given by $xw \mapsto x\, g_x(w)$ for all $x \in \{0,1\}$ and $w \in \{0,1\}^{\ast}$.
Together with the rooted ones, automorphisms of this form can be used to decompose every automorphism of $\T$ as follows.

\begin{definition}\label{def:wreath-decomposition}
Let $\alpha \in \Aut(\T)$ and let $v \in \{0,1\}^{\ast}$ be a vertex.
The \emph{state} of $\alpha$ at $v$ is the unique automorphism $\alpha_v \in \Aut(\T)$ that satisfies
\[
\alpha(vw) = \alpha(v)\, \alpha_v(w)
\]
for every $w \in \{0,1\}^{\ast}$.
For a vertex $x \in \{0,1\}$ of the first level, this reads $\alpha(xw) = \sigma_{\alpha}(x)\, \alpha_x(w)$.
The resulting decomposition $\alpha = \sigma_{\alpha} \circ (\alpha_0, \alpha_1)$ is called the \emph{wreath decomposition} of $\alpha$.
\end{definition}

The wreath decomposition endows us with an isomorphism
\[
\Aut(\T) \longrightarrow \Sym(\{0,1\}) \ltimes \bigl(\Aut(\T) \times \Aut(\T)\bigr),
\ \alpha \mapsto \sigma_{\alpha} \cdot (\alpha_0, \alpha_1),
\]
which we will use to identify an element $\alpha \in \Aut(\T)$ with its wreath decomposition $\sigma_{\alpha} \cdot (\alpha_0, \alpha_1)$. 

\begin{definition}\label{def:self-similar}
A subgroup $G \leq \Aut(\T)$ is called \emph{self-similar} if for every $g \in G$ and every vertex $v \in \{0,1\}^{\ast}$ the state $g_v$ is contained in $G$.
\end{definition}

A prominent example of a self-similar group is the first Grigorchuk group $\mathcal{G}$, see e.g.~\cite{BartholdiGrigorchukSunik03}.
It is the subgroup of $\Aut(\T)$ generated by the four automorphisms $a, b, c, d$ defined as follows.
The generator $a$ is the rooted automorphism corresponding to the non-trivial permutation in $\Sym(\{0,1\})$, while $b, c, d$ fix the first level and are given by the wreath decompositions
\[
b = (a, c),
\qquad
c = (a, d),
\qquad
d = (1, b).
\]
A direct verification shows that each of the generators in $S \defeq \{a, b, c, d\}$ is an involution, so that $\mathcal{G}$ is generated by $S$ as a monoid.
For each $n \in \N_0$, we write
\[
a_n, b_n, c_n, d_n \in \Sym(L_n)
\]
to denote the images of $a, b, c, d$ under $\pi_n$ and we set
\begin{equation}\label{eq:grigorchuk-quotient}
\mathcal{G}_n \defeq \langle a_n, b_n, c_n, d_n \rangle \leq \Sym(L_n).
\end{equation}
The action of $\mathcal{G}$ on $\T$ is \emph{spherically transitive}, that is, the induced action of $\mathcal{G}$ on $L_n$ is transitive for every $n \in \N_0$, see e.g.~\cite[Section~1]{BartholdiGrigorchukSunic2003}.
In other words, the (labeled) Schreier graph of the action of $\mathcal{G}$ on each level $L_n$ is connected.
Recall that, given a group $H$ acting on a set $\Omega$ and a generating set $S$ of $H$, the associated \emph{labeled Schreier graph} is the graph with vertex set $\Omega$ in which, for every $\omega \in \Omega$ and every $s \in S$, there is an edge labeled $s$ from $\omega$ to $\omega s$.
For each $n \in \N_0$, we write $\Gamma_n$ for the labeled Schreier graph of the action of $\mathcal{G}$ on $L_n$ with respect to $S$.
For the rest of the paper we fix the two vertices
\begin{equation}\label{eq:distinguished-vertices}
\rho_n \defeq 1^n
\qquad\text{and}\qquad
\eta_n \defeq 1^{n-1}0
\end{equation}
of $\Gamma_n$.
As the action of $\mathcal{G}$ on the tree $\T$ extends continuously to its boundary $\partial \T = \{0,1\}^{\N} = \mathfrak{C}$, we may also consider the labeled Schreier graph $\Gamma_{\infty}$ of the orbit of the boundary point $\rho \defeq 1^{\infty} \in \mathfrak{C}$, with respect to $S$.
We now relate the Schreier graphs $\Gamma_n$ and $\Gamma_{\infty}$.
In order to do so, we work with a useful description of the graphs $\Gamma_n$ that was provided by Bartholdi and Grigorchuk~\cite[Section~5.1]{BartholdiGrigorchuk00}.
They showed that for every $n \in \N$ the graph $\Gamma_{n+1}$ arises from $\Gamma_n$ in two steps.
First, each vertex $v$ of $\Gamma_n$ is relabeled to $1v$, while the edge labels $b$, $c$, $d$ are replaced by $d$, $b$, $c$, respectively.
Then each $a$-labeled edge, which now connects two vertices $1v$ and $1w$, is subdivided into a path of length three.
More precisely, the edge is removed and two new vertices $0v$ and $0w$ are added.
These are connected to each other by two edges labeled $b$ and $c$, and to $1v$ and $1w$, respectively, by $a$-labeled edges.
In addition, each of the two new vertices carries a $d$-labeled loop.
The two steps are illustrated in~\cref{fig:gammatransition}.
\begin{figure}
\centering

\begingroup
\def\vertexgap{1.25cm}
\def\rowgap{2.5cm}
\begin{tikzpicture}[
    font=\small,
    vertex/.style={circle, fill, inner sep=0pt, minimum size=3pt,
                   outer sep=1pt},
    word/.style={label distance=3pt},
    edge/.style={line width=0.6pt},
    loop edge/.style={edge, loop above, looseness=7, min distance=7mm},
    transition/.style={-{Stealth}, line width=0.8pt, decorate,
        decoration={snake, amplitude=0.5mm, segment length=2.5mm,
                    post length=1.5mm}},
    every label/.style={font=\small}
]
\node[vertex, word, label=below:$11$] (g2-11) {};
\node[vertex, word, right=\vertexgap of g2-11, label=below:$01$] (g2-01) {};
\node[vertex, word, right=\vertexgap of g2-01, label=below:$00$] (g2-00) {};
\node[vertex, word, right=\vertexgap of g2-00, label=below:$10$] (g2-10) {};
\node[left=5mm of g2-11] {$\Gamma_2$};

\draw[edge] (g2-11) -- node[above] {$a$} (g2-01);
\draw[edge] (g2-01) -- node[above] {$b$} node[below] {$c$} (g2-00);
\draw[edge] (g2-00) -- node[above] {$a$} (g2-10);
\foreach \word/\generators in {11/{b\,c\,d},01/d,00/d,10/{b\,c\,d}} {
    \draw[loop edge] (g2-\word) to node[above] {$\generators$} (g2-\word);
}

\node[vertex, word, below=\rowgap of g2-11, label=below:$111$] (step1-111) {};
\node[vertex, word, right=\vertexgap of step1-111, label=below:$101$] (step1-101) {};
\node[vertex, word, right=\vertexgap of step1-101, label=below:$100$] (step1-100) {};
\node[vertex, word, right=\vertexgap of step1-100, label=below:$110$] (step1-110) {};

\draw[edge] (step1-111) -- node[above] {$a$} (step1-101);
\draw[edge] (step1-101) -- node[above] {$b$} node[below] {$d$} (step1-100);
\draw[edge] (step1-100) -- node[above] {$a$} (step1-110);
\foreach \word/\generators in {111/{b\,c\,d},101/c,100/c,110/{b\,c\,d}} {
    \draw[loop edge] (step1-\word) to node[above] {$\generators$} (step1-\word);
}

\node[vertex, word, below=\rowgap of step1-101, label=below:$101$] (g3-101) {};
\node[vertex, word, left=\vertexgap of g3-101, label=below:$001$] (g3-001) {};
\node[vertex, word, left=\vertexgap of g3-001, label=below:$011$] (g3-011) {};
\node[vertex, word, left=\vertexgap of g3-011, label=below:$111$] (g3-111) {};
\node[vertex, word, right=\vertexgap of g3-101, label=below:$100$] (g3-100) {};
\node[vertex, word, right=\vertexgap of g3-100, label=below:$000$] (g3-000) {};
\node[vertex, word, right=\vertexgap of g3-000, label=below:$010$] (g3-010) {};
\node[vertex, word, right=\vertexgap of g3-010, label=below:$110$] (g3-110) {};
\node[left=5mm of g3-111] {$\Gamma_3$};

\foreach \from/\to in {110/010,000/100,101/001,011/111} {
    \draw[edge] (g3-\from) -- node[above] {$a$} (g3-\to);
}
\foreach \from/\to/\lowerlabel in {010/000/c,100/101/d,001/011/c} {
    \draw[edge] (g3-\from) -- node[above] {$b$}
                                  node[below] {$\lowerlabel$} (g3-\to);
}
\foreach \word/\generators in {
    110/{b\,c\,d},010/d,000/d,100/c,101/c,001/d,011/d,111/{b\,c\,d}
} {
    \draw[loop edge] (g3-\word) to node[above] {$\generators$} (g3-\word);
}

\coordinate (top centre) at ($(g2-01)!0.5!(g2-00)$);
\coordinate (middle centre) at ($(step1-101)!0.5!(step1-100)$);
\coordinate (bottom centre) at ($(g3-100)!0.5!(g3-101)$);
\draw[transition] ($(top centre)+(0,-7mm)$)
    -- node[right=1mm] {Step 1} ($(middle centre)+(0,12mm)$);
\draw[transition] ($(middle centre)+(0,-7mm)$)
    -- node[right=1mm] {Step 2} ($(bottom centre)+(0,12mm)$);
\end{tikzpicture}
\endgroup

\caption{The two steps that turn $\Gamma_2$ into $\Gamma_{3}$.
As every element of $S$ is an involution, each $s$-labeled edge is either a loop or one of two mutually reverse edges, which we draw as a single undirected edge.}
\label{fig:gammatransition}
\end{figure}
Starting from $\Gamma_1$, which consists of a single $a$-labeled edge joining $\eta_1=0$ and $\rho_1=1$, together with $b$-, $c$-, and $d$-labeled loops at both vertices, an inductive application of these two steps shows that $\Gamma_n$ is a path (with loops and parallel edges) on $2^n$ vertices whose endpoints are $\eta_{n}$ and $\rho_{n}$.
This yields the following well-known fact, which we record for later reference.

\begin{lemma}\label{lem:separation}
For every $n\geq 1$, the distance between $\eta_n$ and $\rho_n$ in the Schreier graph $\Gamma_n$ is 
\[
d_{\Gamma_n}(\eta_n,\rho_n)=2^n-1.
\]
\end{lemma}
\begin{proof}
Loops and parallel edges do not affect graph distances, so the claim follows from the fact that $\eta_{n}$ and $\rho_{n}$ are the endpoints of $\Gamma_n$.
\end{proof}

\noindent For $g\in\mathcal{G}$, let $|g|_S$ denote its word length with respect to $S$.
As a consequence of~\cref{lem:separation}, we can determine up to which value of $|g|_S$ the element $g$ fixes $\rho_n$ if and only if it fixes $\rho$.

\begin{lemma}\label{lem:fixation-transfer}
Let $n \in \N$ and let $g \in \mathcal{G}$ be an element with $|g|_S \leq 2^{n+1} - 2$.
Then $g$ fixes the vertex $\rho_n$ if and only if $g$ fixes the boundary point $\rho$.
\end{lemma}
\begin{proof}
If $g$ fixes $\rho = 1^{\infty}$, then it fixes each of its prefixes $\rho_n = 1^n$.
Suppose now that $g$ fixes $\rho_n$ but not $\rho$.
Since $g$ acts on $\T$ by automorphisms, there is a maximal $m \in \N$ with $g(1^m) = 1^m$ and $m \geq n$.
Moreover we have $g(1^{m+1}) \in \{1^m 0,\ 1^{m+1}\}$, so the maximality of $m$ gives us $g(\rho_{m+1}) = 1^m 0 = \eta_{m+1}$.
Every word over $S$ representing $g$ therefore yields a path from $\rho_{m+1}$ to $\eta_{m+1}$ in $\Gamma_{m+1}$.
~\cref{lem:separation} therefore implies
\[
|g|_S \geq d_{\Gamma_{m+1}}(\rho_{m+1}, \eta_{m+1}) = 2^{m+1} - 1 \geq 2^{n+1} - 1,
\]
which contradicts our assumption.
\end{proof}

Using~\cref{lem:fixation-transfer}, we will deduce that the balls of radius $2^n - 2$ around the points $\rho_n$ and $\eta_n$ in $\Gamma_n$ are isomorphic, as labeled graphs, to the ball of the same radius around $\rho$ in $\Gamma_{\infty}$.
For a labeled graph $G$, a vertex $p$ of $G$, and $r \in \N$, we will write $G(p, r)$ for the labeled subgraph induced by the vertices at distance at most $r$ from $p$.
To prove the claim, we first record a general criterion under which two such balls are isomorphic.

\begin{lemma}\label{lem:stabilizer-ball-criterion}
Let $H$ be a group with generating set $S$ that acts on two sets $\Omega$ and~$\Omega'$.
Let $G$ and $G'$ be the associated Schreier graphs in which we consider two points $\omega \in \Omega$ and $\omega' \in \Omega'$.
If the stabilizers of $\omega$ and $\omega'$ satisfy 
\[
\St_H(\omega) \cap B_S(2r+1) = \St_H(\omega') \cap B_S(2r+1),
\]
then the map
\[
\varphi \colon G(\omega, r) \rightarrow G'(\omega', r),\ \omega h \mapsto \omega' h
\qquad (|h|_S\leq r)
\]
is a well-defined isomorphism of labeled graphs.
\end{lemma}
\begin{proof}
For any $g, g' \in H$ with $|g|_S, |g'|_S \leq r$, the element $g' g^{-1}$ lies in $B_S(2r)$.
Our hypothesis therefore gives us
\[
\omega g = \omega g'
\iff
g' g^{-1} \in \St_H(\omega)
\iff
g' g^{-1} \in \St_H(\omega')
\iff
\omega' g = \omega' g'.
\]
In particular we see that $\varphi$ is well-defined and injective.
Since every vertex of $G'(\omega', r)$ is of the form $\omega' g$ with $|g|_S \leq r$ it follows that $\varphi$ is also surjective.
To check that $\varphi$ preserves labeled edges in both directions, let $g,h\in H$ satisfy $|g|_S,|h|_S\leq r$, and let $s\in S$.
Since $|gsh^{-1}|_S\leq 2r+1$, our hypothesis gives
\[
\omega gs=\omega h
\iff
gsh^{-1}\in\St_H(\omega)
\iff
gsh^{-1}\in\St_H(\omega')
\iff
\omega'gs=\omega'h.
\]
Thus there is an $s$-labeled edge from $\omega g$ to $\omega h$ if and only if there is an $s$-labeled edge from $\omega'g$ to $\omega'h$, including when the edge is a loop.
Hence $\varphi$ is an isomorphism of labeled graphs.
\end{proof}

\begin{lemma}\label{lem:local-convergence}
For every $n \geq 2$ the map
\[
\Gamma_n(\rho_n, 2^n - 2) \rightarrow \Gamma_{\infty}(\rho, 2^n - 2), \ \rho_n h \mapsto \rho h
\]
is an isomorphism of labeled graphs.
\end{lemma}
\begin{proof}
Let $g \in \mathcal{G}$ be an element with $|g|_S \leq 2 \cdot (2^n - 2)+1$.
By~\cref{lem:stabilizer-ball-criterion}, it suffices to show that $g$ fixes $\rho_n$ if and only if it fixes $\rho$.
Since
\[
2 \cdot (2^n - 2)+1 = 2^{n+1} - 3 \leq 2^{n+1} - 2,
\]
this is precisely the statement of~\cref{lem:fixation-transfer}.
\end{proof}

\begin{corollary}\label{cor:eta-local-convergence}
For every $n \geq 2$ the map
\[
\Gamma_n(\eta_n, 2^{n}-2) \rightarrow \Gamma_{\infty}(\rho, 2^{n}-2), \ \eta_n h \mapsto \rho h
\]
is an isomorphism of labeled graphs.
\end{corollary}
\begin{proof}
Since $\mathcal{G}$ acts on $\T$ by automorphisms, it follows that $\St_{\mathcal{G}}(\rho_n) = \St_{\mathcal{G}}(\eta_n)$.
In particular we have
\[
\St_{\mathcal{G}}(\eta_n) \cap B_S(2r_n+1) = \St_{\mathcal{G}}(\rho_n) \cap B_S(2r_n+1)
\]
for $r_n = 2^{n}-2$.
Thus~\cref{lem:stabilizer-ball-criterion} provides us with the isomorphism
\[
\Gamma_n(\eta_n, r_n) \to \Gamma_n(\rho_n, r_n),\ \eta_n h \mapsto \rho_n h.
\]
Composing it with the isomorphism $\Gamma_n(\rho_n, r_n) \to \Gamma_{\infty}(\rho, r_n)$ of~\cref{lem:local-convergence} gives the claimed isomorphism.
\end{proof}

In what follows we will need a strengthening of spherical transitivity of $\mathcal{G}$, which is provided in \cite[Proposition~12 in~Appendix~A2]{bekka+harpe+grigorchuk}.
To formulate it, we consider the metric $\delta$ on $L_n$ given by $\delta(x,y) = e^{-\ell(x,y)}$ for $x\ne y$, where $\ell(x,y)$ denotes the length of the maximal common prefix of $x$ and $y$.

\begin{lemma}\label{lem:pair-transitivity}
For every $n \in \N_0$, the diagonal action of $\mathcal{G}_n$ on $L_n \times L_n$ is transitive on the pairs of points of the same $\delta$-distance.
\end{lemma}

We now turn to growth results concerning $\mathcal{G}$ and its Schreier graphs $\Gamma_{\infty}$ and $\Gamma_n$.
All growth estimates that enter our arguments will be expressed in terms of a single constant $\alpha_0$, which we now introduce.

\begin{notation}\label{not:alpha0}
Let $\eta$ denote the unique real root of the polynomial $x^3 + x^2 + x - 2$ and set
\[
\alpha_0 \defeq \frac{\log(2)}{\log(2/\eta)} \approx 0.7674.
\]
\end{notation}

The following theorem of Bartholdi~\cite{Bartholdi98} sharpens Grigorchuk's result~\cite{Grigorchuk84} that $\mathcal{G}$ is of subexponential growth.

\begin{theorem}[Bartholdi]\label{thm:G-subexponential}
There is a constant $C > 0$ such that
\[
\gamma_{\mathcal{G}}^S(n) \leq \exp(C n^{\alpha_0})
\]
for every $n \in \N$.
\end{theorem}

\begin{remark}\label{rem:alpha0-optimal}
The constant $\alpha_0$ in \cref{thm:G-subexponential} was shown to be optimal by Erschler and Zheng~\cite{ErschlerZheng20} in the sense that $\exp(n^{\beta}) \preceq \gamma_{\mathcal{G}}^S(n)$ for every $\beta < \alpha_0$.
\end{remark}

While~\cref{thm:G-subexponential} concerns the growth of balls in the Cayley graph $\Cay(\mathcal{G},S)$, the quantity we are interested in for the Schreier graphs $\Gamma$ and $\Gamma_n$ is the \emph{inverted orbit growth} -- a notion systematically studied by Bartholdi and Erschler~\cite{BartholdiErschler12}.

\begin{convention}
Following their convention in~\cite{BartholdiErschler12}, we will from now on work with right actions, which are more convenient for the study of inverted orbits.
\end{convention}

\begin{definition}\label{def:inverted-orbit}
Let $H$ be a group acting from the right on a set $\Omega$. Let $A$ be a generating set of $H$, and let $\omega \in \Omega$.
The \emph{inverted orbit} of a word $w = s_1 s_2 \cdots s_{\ell} \in A^{\ast}$ at $\omega$ is the set
\[
\OO_{\omega}(w) \defeq \{\omega,\ \omega s_{\ell},\ \omega s_{\ell-1}s_{\ell},\ \ldots,\ \omega s_1 s_2 \cdots s_{\ell}\} \subseteq \Omega.
\]
The cardinality of $\OO_{\omega}(w)$ will be denoted by $\delta_{\omega}(w)$.
We further define the \emph{inverted orbit growth function}
\[
\Delta_{\omega}^{A}(n) \defeq \max\Set{\delta_{\omega}(w)}{w \in A^{\ast},\ |w| \leq n},
\]
and we write
\[
N_{\omega}^{A}(n) \defeq \bigl| \Set{\OO_{\omega}(w)}{w \in A^{\ast},\ |w| \leq n} \bigr|
\]
for the number of distinct inverted orbits arising from words of length at most $n$.
\end{definition}

Bartholdi and Erschler studied the functions $\Delta_{\rho}^{S}$ and $N_{\rho}^{S}$ for the action of $\mathcal{G}$ on the $\mathcal{G}$-orbit of $\rho$.
In~\cite[Proposition~4.4 and Lemma~4.9]{BartholdiErschler12} they obtained the following upper bounds for these functions, which involve the same exponent $\alpha_0$ as in \cref{thm:G-subexponential}.

\begin{theorem}[Bartholdi-Erschler]\label{thm:BE}
There exists a constant $C > 0$ such that
\[
\Delta_{\rho}^{S}(n) \leq C n^{\alpha_0}
\qquad\text{and}\qquad
N_{\rho}^{S}(n) \leq \exp(C n^{\alpha_0})
\]
for every $n \in \N$, where $\Delta_{\rho}^{S}$ and $N_{\rho}^{S}$ refer to the action of $\mathcal{G}$ on the $\mathcal{G}$-orbit of $\rho$.
\end{theorem}

Using our preceding lemmas, we transfer the estimates from~\cref{thm:BE} to the action of $\mathcal{G}$ on $L_n$ and the points $\rho_n$ and $\eta_n$.

\begin{proposition}\label{prop:BE-finite}
Let $C$ be as in~\cref{thm:BE} and let $r_n = 2^{n}-2$.
For every $n$, every $\xi_n \in \{\rho_n, \eta_n\}$, and every $k \leq r_n$, we have
\[
\Delta_{\xi_n}^{S_n}(k) \leq C k^{\alpha_0}
\qquad\text{and}\qquad
N_{\xi_n}^{S_n}(k) \leq \exp(C k^{\alpha_0}),
\]
where $\Delta_{\xi_n}^{S_n}$ and $N_{\xi_n}^{S_n}$ refer to the action of $\mathcal{G}_n$ on $L_n$ with respect to $S_n$.
\end{proposition}
\begin{proof}
We fix $\xi_n \in \{\rho_n, \eta_n\}$.
By~\cref{lem:local-convergence} and~\cref{cor:eta-local-convergence}, the map
\[
\varphi \colon \Gamma_n(\xi_n, r_n) \rightarrow \Gamma_{\infty}(\rho, r_n),\ \xi_n h \mapsto \rho h
\]
is a well-defined isomorphism of labeled graphs.
Consider a word $w = s_1 \ldots s_{\ell} \in S_n^{\ast}$ with $|w| = \ell \leq r_n$.
Every element of the inverted orbit $\OO_{\xi_n}(w)$ is of the form $\xi_n s_i s_{i+1} \ldots s_{\ell}$ and hence lies in the ball $\Gamma_n(\xi_n, r_n)$.
By the same reasoning we have $\OO_{\rho}(w) \subseteq \Gamma_{\infty}(\rho, r_n)$.
Under the isomorphism $\varphi$, the vertex $\xi_n s_i \ldots s_{\ell}$ corresponds to $\rho s_i \ldots s_{\ell}$, so the isomorphism restricts to a bijection $\OO_{\xi_n}(w) \to \OO_{\rho}(w)$.
In particular we have $\delta_{\xi_n}(w) = \delta_{\rho}(w)$, and distinct inverted orbits at $\xi_n$ correspond to distinct inverted orbits at $\rho$, which completes the proof.
\end{proof}

\section{Inverted orbits on the Cantor set}\label{sec:inverted-orbits-cantor}

We now simulate the inverted orbits of the previous section by homeomorphisms of a Cantor set.
To this end, we fix an integer $n \in \N$ and consider the Higman--Thompson group $V_{Y_n}$ corresponding to the Cantor set
\[
  \mathfrak{C}_{Y_n} = \mathfrak{C} \times Y_n,
\]
where $Y_n \defeq L_n \times \{1,2,3,4\}$.
Thus $\mathfrak{C}_{Y_n}$ consists of $|Y_n| = 2^{n+2}$ copies of $\mathfrak{C}$, indexed by $Y_n$.
For $y \in Y_n$ we write $\mathfrak{C}_y \defeq \mathfrak{C} \times \{y\}$ for the $y$-th copy.
We think of $Y_n$ as four disjoint sheets of $L_n$ and write $\pr_n \colon Y_n \to L_n$ for the projection onto the first coordinate.
Let us now introduce four types of elements in $V_{Y_n}$, which will be shown in~\cref{sec:generating-V} to generate $V_{Y_n}$.
The generators of types~1, 2 and~3 permute the copies $\mathfrak{C}_y$ among each other, while acting as the identity on the Cantor coordinate.
In particular, these generators preserve the partition $\{\mathfrak{C}_y\}_{y \in Y_n}$.
The generators of type~4 do not preserve this partition, but they still lie in $V_{Y_n}$, being dyadic homeomorphisms of $\mathfrak{C}_{Y_n}$ that correspond to the generators $X_0$ and $X_1$ of Thompson's group $F$.

\medskip
\noindent\textbf{Type 1: The truncated Grigorchuk generators.}
We identify each element $g \in S_n = \{a_n, b_n, c_n, d_n\}$ with the homeomorphism of $\mathfrak{C}_{Y_n}$ that acts diagonally on the four sheets, i.e.\ it maps $\mathfrak{C}_{(x,i)}$ to $\mathfrak{C}_{(xg,i)}$, leaving the Cantor coordinate unchanged.

\medskip
\noindent\textbf{Type 2: The symmetric group over $\eta_n$.}
We let the symmetric group $\Sym(4)$ permute the four copies of $\mathfrak{C}$ corresponding to the set $\{\eta_n\} \times \{1,2,3,4\}$ and act as the identity outside of $\mathfrak{C} \times \{\eta_n\} \times \{1,2,3,4\}$.
The resulting copy of $\Sym(4)$ in $\Homeo(\mathfrak{C}_{Y_n})$ will be denoted by $\Sym(4)_{\eta_n}$.

\medskip
\noindent\textbf{Type 3: The linking transposition.}
Let $\theta_n \defeq \rho_n a_n = 0\,1^{n-1}$ and let $\tau_n \in V_{Y_n}$ be the homeomorphism that transposes the two copies $\mathfrak{C}_{(\rho_n,1)}$ and $\mathfrak{C}_{(\theta_n,1)}$, while being the identity on their Cantor coordinates and outside these copies.

\medskip
\noindent\textbf{Type 4: The two Thompson generators.}
Recall the standard generators $X_0, X_1$ of Thompson's group $\F$ from~\cref{sec:V-background}.
We place copies of them on $\mathfrak{C} \times \{\rho_n\} \times \{2,3,4\}$ as follows.
First, let $X_1^{(n)}$ be the homeomorphism that is supported on $\mathfrak{C}_{(\rho_n,4)}$ and acts there as $X_0$ under the canonical identification $\mathfrak{C}_{(\rho_n,4)} \cong \mathfrak{C}$.
Second, we identify the union
\[
  D_n \defeq \mathfrak{C}_{(\rho_n,2)} \sqcup \mathfrak{C}_{(\rho_n,3)}
\]
with $\mathfrak{C}$ via
\[
  D_n \to \mathfrak{C},\ (\xi, (\rho_n, i)) \mapsto (i-2)\,\xi,
\]
and let $X_0^{(n)}$ be the homeomorphism that is supported on $D_n$ and acts there as $X_0$ under this identification.
The naming reflects the role these generators will play in~\cref{sec:generating-V}.
Although $X_1^{(n)}$ acts as $X_0$ on its own sheet, once it is transported onto the right half of the support of $X_0^{(n)}$, it acts there as $X_0$ on a right half, which is precisely the generator $X_1$ of $\F$.

\begin{figure}
\begin{tikzpicture}[
    scale=0.9,
    every node/.style={font=\small},
    interval/.style={black, line width=0.8pt},
    tick/.style={black, line width=0.6pt},
    diag/.style={gray!60, line width=0.45pt, <->, >=Stealth},
    symbox/.style={draw=gray!65, fill=gray!20, line width=0.7pt},
    taubox/.style={draw=black, fill=gray!80, line width=0.7pt},
    xbox/.style={draw=black, fill=gray!30, line width=0.7pt},
    mainarrow/.style={
    black,
    line width=0.65pt,
    {Stealth[length=1.4mm,width=0.9mm]}-{Stealth[length=1.4mm,width=0.9mm]}
},
    smallarrow/.style={
    black,
    line width=0.65pt,
    -{Stealth[length=1.4mm,width=0.9mm]}
}
]

\foreach \y in {0,-1,-2,-3} {
    \draw[interval] (0,\y) -- (8,\y);

    \foreach \x in {0,...,8} {
       \draw[tick] (\x,\y-0.15) -- (\x,\y+0.15);
    }
}

\foreach \x/\word in {
    0/000,
    1/001,
    2/010,
    3/011,
    4/100,
    5/101,
    6/110,
    7/111
} {
    \node[above] at (\x+0.5,0.35) {\scriptsize $\word$};
}

\foreach \y in {0,-1,-2,-3} {
    \draw[diag] (0.5,\y+0.25) to[bend left=10] (4.5,\y+0.25);
    \draw[diag] (1.5,\y+0.18) to[bend left=10] (5.5,\y+0.18);
    \draw[diag] (2.5,\y+0.11) to[bend left=10] (6.5,\y+0.11);
    \draw[diag] (3.5,\y+0.04) to[bend left=10] (7.5,\y+0.04);
}

\node[
    rotate=90,
    align=center
] at (-0.75,-1.5)
    {$S_n$ acts diagonally\\on the four levels.};

\foreach \y in {0,-1,-2,-3} {
    \draw[symbox] (6.1,\y-0.28) rectangle (6.9,\y+0.28);
}

\node[
    fill=white,
    inner sep=1.5pt
] at (6.6,-1.5)
    {\tiny $\operatorname{Sym}(4)_{\eta_n}$};

\draw[smallarrow]
    (6.1,0) to[out=180,in=180] (6.1,-1);
\draw[smallarrow]
    (6.1,-1) to[out=180,in=180] (6.1,-2);
\draw[smallarrow]
    (6.1,-2) to[out=180,in=180] (6.1,-3);

\node[gray!65!black, below] at (6.5,-3.45)
    {$\eta_n$};

\draw[taubox] (3.1,-0.3) rectangle (3.9,0.3);
\draw[taubox] (7.1,-0.3) rectangle (7.9,0.3);

\draw[mainarrow]
    (3.5,0.32) to[bend left=32] (7.5,0.32);

\node[black, above] at (5.5,1.15)
    {$\tau_n$};

\node[black, below] at (3.5,-3.45)
    {$\theta_n$};

\node[black, below] at (7.5,-0.45)
    {};

\foreach \y in {-1,-2,-3} {
    \draw[xbox] (7.1,\y-0.3) rectangle (7.9,\y+0.3);
}

\node[black] at (8.3,-1.5)
    {$X_0^{(n)}$};

\node[black] at (8.4,-3)
    {$X_1^{(n)}$};

\draw[smallarrow]
    (7.5,-1) to[bend right=45] (7.5,-2);

\draw[smallarrow]
    (7.5,-2) to[bend right=45] (7.5,-1);

\draw[smallarrow]
    (7.65,-3) arc[start angle=0, end angle=355, radius=0.18];


\node[black, below] at (7.5,-3.45)
    {$\rho_n$};

\end{tikzpicture}
\caption{Illustration of the generators in $T_n$ and their supports in $\CC\times Y_n=(\CC\times L_n)\times\{1,2,3,4\}$ for $n=3$. The diagonal action of the Grigorchuk generators in $S_n$ is indicated by the various slightly bent arrows in light gray. The generator $X_0^{(n)}$ acts on the union of the rightmost Cantor sets on levels $2$ and $3$. This is indicated by the pair of arrows to its left; note that $X_0^{(n)}$ does not simply interchange these two Cantor sets. 
For $n=3$ we have $\theta_n=011\in L_n$ and $\eta_n=110\in L_n$ and $\rho_n=111\in L_n$.}
\label{figure: support of generators}
\end{figure}

\begin{notation}\label{not:T_n}
The union of the four types of generators defined above will be denoted by
\[
  T_n \defeq S_n \cup \Sym(4)_{\eta_n} \cup \{\tau_n\} \cup \{(X_0^{(n)})^{\pm 1}, (X_1^{(n)})^{\pm 1}\},
\]
and we write $W_n \defeq \langle T_n \rangle \leq V_{Y_n}$ for the group it generates.
\end{notation}

The four types of generators in $T_n$ are illustrated in~\cref{figure: support of generators}.

\subsection*{Commuting conjugates}

If a homeomorphism $\sigma$ is supported on $\mathfrak{C}_Z$ for some $Z \subseteq Y_n$, then for $g \in S_n^{\ast}$ the conjugate $\sigma^{g} \defeq g^{-1}\sigma g$ is supported on $\mathfrak{C}_{Zg}$.
Applying this to the case where $\sigma$ is a generator of type~2, 3, or 4 we see that the $L_n$-projection of such a conjugate is given by
\[
  \pr_n(\supp(\sigma^{g})) =
  \begin{cases}
    \{\eta_n g\} & \text{for type 2},\\
    \{\theta_n g,\rho_n g\} & \text{for type 3},\\
    \{\rho_n g\} & \text{for type 4}.
  \end{cases}
\]

Recall the value $r_n = 2^n - 2$ from \cref{sec:Grigorchuk-background}.
The following lemma provides us with sufficient conditions under which conjugates of generators of type~2, 3, and 4 by words of length less than $r_n/2$ in $S_n$ commute.

\begin{lemma}\label{lem:commuting-conjugates}
Let $n \in \N$ and let $g, g' \in S_n^{\ast}$ be words of length $|g|_{S_n}, |g'|_{S_n} < r_n/2$.
If $\sigma$ and $\sigma'$ are generators of distinct types from $\{2,3,4\}$, then the conjugates $\sigma^{g}$ and $(\sigma')^{g'}$ commute.
\end{lemma}
\begin{proof}
We will argue by showing that the supports of $\sigma^{g}$ and $(\sigma')^{g'}$ are disjoint so that these elements commute.
Suppose first that $\{\sigma,\sigma'\}$ involves types $3$ and $4$.
Then the supports lie on disjoint sheets, namely sheet $1$ for type~3 and sheets $2,3,4$ for type~4.
Since $S_n$ preserves each sheet, the conjugates remain on disjoint sheets and their supports are disjoint.
Suppose now that $\{\sigma,\sigma'\}$ involves type~2 together with type~3 or type~4.
Here we compare the $L_n$-projections of the two supports.
The type-2 support projects to $\{\eta_n g'\}$ while the type-3 and type-4 supports project to $\{\theta_n g,\rho_n g\}$ and $\{\rho_n g\}$, respectively.
Thus it suffices to show that $\xi_n g \neq \eta_n g'$ for $\xi_n \in \{\theta_n, \rho_n\}$.
Otherwise we would obtain
\begin{align*}
d_{\Gamma_n}(\rho_n, \eta_n)
&\leq d_{\Gamma_n}(\rho_n, \xi_n)
+ d_{\Gamma_n}(\xi_n, \xi_n g)
+ d_{\Gamma_n}(\xi_n g, \eta_n g')
+ d_{\Gamma_n}(\eta_n g', \eta_n)\\
&\leq 1 + |g|_{S_n} + 0 + |g'|_{S_n}\\
&< 1 + r_n/2 + 0 + r_n/2\\
&= r_n +1\\
&= 2^n - 1.
\end{align*}
On the other hand, \cref{lem:separation} tells us that $d_{\Gamma_n}(\rho_n, \eta_n) = 2^n - 1$.
In particular we see that the supports of $\sigma^{g}$ and $(\sigma')^{g'}$ are disjoint.
\end{proof}

\subsection*{A normal form}

Using~\cref{lem:commuting-conjugates}, we obtain a normal form for
words over $T_n$ of length less than $r_n/2$.

\begin{lemma}\label{lem:normal-form}
Let $n \in \N$ and let $w$ be a word over $T_n$ of length $\ell < r_n/2$.
The element of $W_n$ represented by $w$ can be written as
\[
p_1 \cdot p_2 \cdot p_3 \cdot p_4,
\]
where $p_1$ is represented by a word of the form $g_1\dots g_{\ell}$ over $S_n \cup \{1\}$ and the remaining factors have the form
\begin{itemize}
\item $p_2 = (a_1^{(2)})^{\,g_1 g_2 \cdots g_{\ell}} \cdot (a_2^{(2)})^{\,g_2 \cdots g_{\ell}} \cdot \ldots \cdot (a_{\ell}^{(2)})^{\,g_{\ell}}$,
\item $p_3 = (a_1^{(3)})^{\,g_1 g_2 \cdots g_{\ell}} \cdot (a_2^{(3)})^{\,g_2 \cdots g_{\ell}} \cdot \ldots \cdot (a_{\ell}^{(3)})^{\,g_{\ell}}$,
\item $p_4 = (a_1^{(4)})^{\,g_1 g_2 \cdots g_{\ell}} \cdot (a_2^{(4)})^{\,g_2 \cdots g_{\ell}} \cdot \ldots \cdot (a_{\ell}^{(4)})^{\,g_{\ell}}$,
\end{itemize}
where for each $i$ we have $g_i \in S_n \cup \{1\}$ and
\[
  a_i^{(2)} \in \Sym(4)_{\eta_n},
  \qquad
  a_i^{(3)} \in \{\tau_n, 1\},
  \qquad
  a_i^{(4)} \in \{(X_0^{(n)})^{\pm 1}, (X_1^{(n)})^{\pm 1}, 1\}.
\]
\end{lemma}
\begin{proof}
Write $w = t_1 t_2 \cdots t_{\ell}$ with each $t_i \in T_n$.
For each $i$, set $g_i = t_i$ and $a_i = 1$ if $t_i \in S_n$, and set $g_i = 1$ and $a_i = t_i$ if $t_i \notin S_n$.
Then we can rewrite $w$ as
\begin{equation}\label{eq:rewriting-w}
w = a_1 g_1 a_2 g_2 \cdots a_{\ell} g_{\ell}
    = g_1 g_2 \cdots g_{\ell} \cdot a_1^{\,g_1 g_2 \cdots g_{\ell}} \, a_2^{\,g_2 \cdots g_{\ell}} \cdots a_{\ell}^{\,g_{\ell}}.
\end{equation}
For each $i$ the letter $a_i$ is trivial or a single generator of type $2$, $3$, or $4$.
Accordingly we set
\[
  a_i^{(j)} =
  \begin{cases}
    a_i & \text{if } a_i \text{ is of type } j,\\
    1 & \text{otherwise,}
  \end{cases}
\]
for each $j \in \{2,3,4\}$.
Thus $a_i = a_i^{(2)} a_i^{(3)} a_i^{(4)}$, where at most one factor is non-trivial.
Since the word length of $g_i \cdots g_{\ell}$ with respect to $S_n$ is bounded above by $\ell < r_n/2$ we can apply~\cref{lem:commuting-conjugates} to deduce that any two of the conjugates $(a_i^{(j)})^{\,g_i \cdots g_{\ell}}$ of distinct types $j$ commute.
By substituting $a_i = a_i^{(2)} a_i^{(3)} a_i^{(4)}$ in~\eqref{eq:rewriting-w} we therefore obtain
\begin{align*}
  w
  &= g_1 \cdots g_{\ell}
     \cdot \prod_{i=1}^{\ell} \bigl(a_i^{(2)} a_i^{(3)} a_i^{(4)}\bigr)^{\,g_i \cdots g_{\ell}}\\
  &= g_1 \cdots g_{\ell}
     \cdot \prod_{i=1}^{\ell} (a_i^{(2)})^{\,g_i \cdots g_{\ell}}
    (a_i^{(3)})^{\,g_i \cdots g_{\ell}}
    (a_i^{(4)})^{\,g_i \cdots g_{\ell}}\\
  &= \underbrace{g_1 \cdots g_{\ell}}_{=: \,p_1}
     \cdot \underbrace{\prod_{i=1}^{\ell} (a_i^{(2)})^{\,g_i \cdots g_{\ell}}}_{=: \,p_2}
     \cdot \underbrace{\prod_{i=1}^{\ell} (a_i^{(3)})^{\,g_i \cdots g_{\ell}}}_{=: \,p_3}
     \cdot \underbrace{\prod_{i=1}^{\ell} (a_i^{(4)})^{\,g_i \cdots g_{\ell}}}_{=: \,p_4},
\end{align*}
which proves the claim.
\end{proof}

\subsection*{Counting the pieces}

We now bound the number of elements of $W_n$ that are representable by a word over $T_n$ of length at most $\ell$, for $\ell < r_n/2$.
By~\cref{lem:normal-form}, every such element is determined by the quadruple $(p_1, p_2, p_3, p_4)$ of factors of its normal form.
For $k \in \{1,2,3,4\}$ we write
\[
  P_k(\ell) \defeq \Set{p_k}{p_k \text{ is the $k$-th factor of the normal form of some } w \in T_n^{\ast},\ |w| \leq \ell}.
\]
Our goal is to provide upper bounds for $|P_k(\ell)|$ for each $k$.
This will be done in the following four lemmas, which involve the constant $\alpha_0$ from~\cref{not:alpha0}.

\begin{lemma}\label{lem:count-p1}
There is a constant $C_1 > 0$ such that $|P_1(\ell)| \leq \exp(C_1 \ell^{\alpha_0})$ for every~$\ell$.
\end{lemma}
\begin{proof}
By~\cref{lem:normal-form}, $p_1$ is represented by a word of the form $g_1\dots g_{\ell}$ over $S_n \cup \{1\}$.
Thus $p_1$ is the image under $\pi_n$ of an element of $\mathcal{G}$ of word length at most $\ell$ over $S$.
From~\cref{thm:G-subexponential} it therefore follows that there is a constant $C_1$ with
\[
|P_1(\ell)| \leq \gamma_{\mathcal{G}}^{S}(\ell) \leq \exp(C_1 \ell^{\alpha_0})
\]
for every $\ell$.
\end{proof}

\begin{lemma}\label{lem:count-p2}
There is a constant $C_2 > 0$ such that $|P_2(\ell)| \leq \exp(C_2 \ell^{\alpha_0})$ for every $n \in \N$ and every $\ell < r_n/2$.
\end{lemma}
\begin{proof}
By~\cref{lem:normal-form}, the factor $p_2$ of a word of length $\ell < r_n/2$ over $T_n$ is a product of the form
\begin{equation}\label{eq:p_2}
(a_1^{(2)})^{g_1 \cdots g_{\ell}} \cdots (a_{\ell}^{(2)})^{g_{\ell}},
\end{equation}
where $g_i \in S_n \cup \{1\}$ and $a_i^{(2)} \in \Sym(4)_{\eta_n}$.
In particular, $p_2$ is a product of type-1 and type-2 generators, which permute the copies $\mathfrak{C}_y$ while acting trivially on the Cantor coordinate.
We can therefore think of $p_2$ as an element of the symmetric group $\Sym(Y_n)$.
Under this identification, the type-1 generators act on the first coordinate of $Y_n = L_n \times \{1,2,3,4\}$ and the type-2 generators permute the subset $\eta_n \times \{1,2,3,4\} \subseteq Y_n$.
Thus the subgroup of $\Sym(Y_n)$ that is generated by the type-1 and type-2 generators is precisely the permutational wreath product
\[
\Sym(4) \wr_{L_n} \mathcal{G}_n = \Bigl( \prod_{x \in L_n} \Sym(4)_x \Bigr) \rtimes \mathcal{G}_n,
\]
where $\Sym(4)_x \defeq \Sym(\{x\} \times \{1,2,3,4\})$.
In this language, the conjugates $(a_i^{(2)})^{g_i \cdots g_{\ell}}$ from~\eqref{eq:p_2} lie in the factor $\Sym(4)_x$ for $x = \eta_n\,(g_i \cdots g_{\ell})$.
The element $p_2$ is therefore contained in the base group $\prod_{x \in L_n} \Sym(4)_x$ and its support, i.e.\ the coordinates $x \in L_n$ at which $p_2$ is non-trivial, is a subset of the inverted orbit
\[
  \OO_{\eta_n}(w) = \{\eta_n,\ \eta_n g_{\ell},\ \eta_n g_{\ell-1}g_{\ell},\ \ldots,\ \eta_n g_1 \cdots g_{\ell}\}
\]
of the word $w = g_1 \cdots g_{\ell}$ at $\eta_n$.
Thus $p_2$ is determined by
one element of $\Sym(4)$ for each $x \in \OO_{\eta_n}(w)$.
By~\cref{prop:BE-finite} there is a constant~$C>0$ such that for all $\ell < r_n/2$ the cardinality of $\OO_{\eta_n}(w)$ is bounded above by $C \ell^{\alpha_0}$.
Moreover,~\cref{prop:BE-finite} tells us that there are at most $\exp(C \ell^{\alpha_0})$ inverted orbits $\OO_{\eta_n}(u)$ where $u$ is a word of length at most~$\ell$ over $S_n$.
Since for each of these inverted orbits there are at most
\[
|\Sym(4)|^{|\OO_{\eta_n}(w)|} \leq \exp(C \ell^{\alpha_0} \log(24))
\]
choices of group elements, we can choose a constant $C_2$ such that
\[
|P_2(\ell)| \leq \exp(C \ell^{\alpha_0}) \cdot \exp(C \ell^{\alpha_0} \log(24)) \leq \exp(C_2 \ell^{\alpha_0})
\]
for all $\ell < r_n/2$.
\end{proof}

\begin{lemma}\label{lem:count-p3}
There is a constant $C_3 > 0$ such that $|P_3(\ell)| \leq \exp(C_3 \ell^{\alpha_0} \log \ell)$ for every $n \in \N$ and every $1 < \ell \leq r_n/3$.
\end{lemma}
\begin{proof}
By~\cref{lem:normal-form}, the factor $p_3$ of a word of length $\ell \leq r_n/3$ over $T_n$ is a product of the form
\begin{equation}\label{eq:p_3}
  (a_1^{(3)})^{g_1 \cdots g_{\ell}} \cdots (a_{\ell}^{(3)})^{g_{\ell}},
\end{equation}
where $g_i \in S_n \cup \{1\}$ and $a_i^{(3)} \in \{\tau_n, 1\}$.
In particular, $p_3$ is a product of type-1 and type-3 generators, which permute the copies $\mathfrak{C}_y$ while acting trivially on the Cantor coordinate.
Since these conjugates act trivially on the Cantor coordinate, we now regard them as the induced permutations of the index set $Y_n$.
The support of the permutation induced by the conjugate $(a_i^{(3)})^{g_i \cdots g_{\ell}}$ from~\eqref{eq:p_3} is either empty or $\{\theta_n g_i \cdots g_{\ell},\rho_n g_i \cdots g_{\ell}\} \times \{1\}$.
We can therefore think of $p_3$ as an element of $\Sym(L_n) \cong \Sym(L_n \times \{1\})$.
In this interpretation, the conjugate $(a_i^{(3)})^{g_i \cdots g_{\ell}}$ is, when non-trivial, the transposition
\[
  (\rho_n\ \theta_n)^{g_i \cdots g_{\ell}} = \bigl(\rho_n\,(g_i \cdots g_{\ell})\ \ \theta_n\,(g_i \cdots g_{\ell})\bigr)
\]
in $\Sym(L_n)$.
Setting $w = g_1 \cdots g_{\ell}$, the element $p_3$ is therefore a permutation whose support is contained in the union of the inverted orbits
\[
\OO_{\rho_n}(w) = \{\rho_n,\ \rho_n g_{\ell},\ \rho_n g_{\ell-1}g_{\ell},\ \ldots,\ \rho_n g_1 \cdots g_{\ell}\}
\]
and
\[
\OO_{\theta_n}(w) = \{\theta_n,\ \theta_n g_{\ell},\ \theta_n g_{\ell-1}g_{\ell},\ \ldots,\ \theta_n g_1 \cdots g_{\ell}\}.
\]
Let us now consider the word
\[
w' \defeq a_n g_1 a_n \cdot a_n g_2 a_n \cdot \ldots \cdot a_n g_{\ell-1} a_n \cdot a_n g_{\ell},
\]
which is of length $3 \ell -1$ over $S_n \cup \{1\}$.
Using that $a_n^2 = 1$, we see that the inverted orbit $\OO_{\rho_n}(w')$ contains the elements
\begin{align*}
& \rho_n a_n g_i a_n \cdot a_n g_{i+1} a_n \cdot \ldots \cdot a_n g_{\ell-1} a_n \cdot a_n g_{\ell}\\
=\ &\theta_n g_i a_n \cdot a_n g_{i+1} a_n \cdot \ldots \cdot a_n g_{\ell-1} a_n \cdot a_n g_{\ell}\\
=\ &\theta_n g_i g_{i+1} \cdot \ldots \cdot g_{\ell-1} g_{\ell}
\end{align*}
for each $i$.
By considering the suffixes of $w'$ starting at $g_i$, we also see that
$\rho_n g_i \cdots g_\ell \in \OO_{\rho_n}(w')$ for every
$1 \leq i \leq \ell$.
Hence the support of $p_3$ is contained in $\OO_{\rho_n}(w')$. 
From~\cref{prop:BE-finite} we know that there is a constant $C > 0$ such that 
\[
|\OO_{\rho_n}(w')| \leq C\,(3\ell)^{\alpha_0}
\]
for $3\ell - 1 \leq r_n$.
Thus $p_3$ lies in the symmetric group $\Sym(\OO_{\rho_n}(w'))$, whose cardinality is bounded above by $\lfloor C\,(3\ell)^{\alpha_0} \rfloor!$.
Moreover,~\cref{prop:BE-finite} tells us that there are at most $\exp\bigl(C (3\ell)^{\alpha_0}\bigr)$ inverted orbits $\OO_{\rho_n}(u)$, where $u$ is a word of length at most $3\ell-1$ over $S_n \cup \{1\}$.
The cardinality of $P_3(\ell)$ can therefore be bounded above by
\begin{equation}\label{eq:P_3-inequality}
|P_3(\ell)| \leq \lfloor C\,(3\ell)^{\alpha_0} \rfloor! \cdot \exp\bigl(C (3\ell)^{\alpha_0}\bigr).
\end{equation}
Using the estimate $m! \leq m^m = \exp(m \log m)$ with $m = \lfloor C\,(3\ell)^{\alpha_0} \rfloor$ we obtain
\[
\lfloor C\,(3\ell)^{\alpha_0} \rfloor!
\leq \exp\bigl(\lfloor C\,(3\ell)^{\alpha_0} \rfloor \cdot \log( \lfloor C\,(3\ell)^{\alpha_0} \rfloor )\bigr).
\]
As the factor $\exp(C (3\ell)^{\alpha_0})$ in~\eqref{eq:P_3-inequality} is dominated by $\lfloor C\,(3\ell)^{\alpha_0} \rfloor!$ we conclude that there is a constant $C_3 > 0$ with
\[
|P_3(\ell)| \leq \exp(C_3 \ell^{\alpha_0} \log \ell)
\]
for $1 < \ell \leq r_n/3$ and all $n$.
\end{proof}

\begin{lemma}\label{lem:count-p4}
There is a constant $C_4 > 0$ such that $|P_4(\ell)| \leq \exp(C_4 \ell^{\alpha_0} \log \ell)$ for every $n \in \N$ and every $1 < \ell < r_n/2$.
\end{lemma}
\begin{proof}
By~\cref{lem:normal-form}, the factor $p_4$ of a word of length $\ell < r_n/2$ over $T_n$ is a product of the form
\begin{equation}\label{eq:p_4}
  (a_1^{(4)})^{g_1 \cdots g_{\ell}} \cdots (a_{\ell}^{(4)})^{g_{\ell}},
\end{equation}
where $g_i \in S_n \cup \{1\}$ and $a_i^{(4)} \in \{(X_0^{(n)})^{\pm 1}, (X_1^{(n)})^{\pm 1}, 1\}$.
The supports of $X_0^{(n)}$ and $X_1^{(n)}$ are contained in $\mathfrak{C}_{\{\rho_n\} \times \{2,3\}}$ and $\mathfrak{C}_{(\rho_n,4)}$, respectively.
In particular, the supports are disjoint and the elements $X_0^{(n)}$, $X_1^{(n)}$ commute.
Since $X_0^{(n)}$ and $X_1^{(n)}$ are moreover of infinite order they generate a subgroup isomorphic to $\Z^2$ whose support lies in $\mathfrak{C}_{\{\rho_n\} \times \{2,3,4\}}$.
As the type-1 generators only act on the first coordinate of $Y_n = L_n \times \{1,2,3,4\}$, we deduce that the group generated by the type-1 and type-4 generators can be identified with the permutational wreath product
\[
  \Z^2 \wr_{L_n} \mathcal{G}_n = \Bigl( \prod_{x \in L_n} A_x \Bigr) \rtimes \mathcal{G}_n ,
\]
where $A_x$ is the copy of $\Z^2$ at the coordinate $x$.
Under this identification, each of the factors $(a_i^{(4)})^{g_i \cdots g_{\ell}}$ from~\eqref{eq:p_4} lies in $A_x$ for $x = \rho_n\,(g_i \cdots g_{\ell})$.
The element $p_4$ is therefore contained in the base group $\prod_{x \in L_n} A_x$ and its support is a subset of the inverted orbit
\begin{equation}\label{eq:p_4-inverted-orbit}
\OO_{\rho_n}(w) = \{\rho_n,\ \rho_n g_{\ell},\ \ldots,\ \rho_n g_1 \cdots g_{\ell}\}
\end{equation}
of $w = g_1 \cdots g_{\ell}$ at $\rho_n$.
Thus $p_4$ is determined by one element $z_x$ of $A_x \cong \Z^2$ for each $x \in \OO_{\rho_n}(w)$, where the word length of $z_x$ with respect to the standard basis of $\Z^2$ is bounded above by $\ell$.
Hence the number of such elements is bounded above by
\[
((2\ell + 1)^2)^{\abs{\OO_{\rho_n}(w)}}.
\]
From~\cref{prop:BE-finite} we know that there is a constant $C > 0$ such that
\[
|\OO_{\rho_n}(w)| \leq C \ell^{\alpha_0}
\]
for every $n$ and every $\ell \leq r_n$.
Moreover,~\cref{prop:BE-finite} tells us that there are at most $\exp(C \ell^{\alpha_0})$ inverted orbits $\OO_{\rho_n}(u)$, where $u$ is a word of length at most $\ell$ over $S_n \cup \{1\}$.
Combining these estimates, we obtain
\[
|P_4(\ell)| \leq ((2\ell + 1)^2)^{C \ell^{\alpha_0}} \cdot \exp(C \ell^{\alpha_0})
= \exp\bigl(2 C \ell^{\alpha_0} \log(2\ell + 1)\bigr) \cdot \exp(C \ell^{\alpha_0}).
\]
Since $\exp(C \ell^{\alpha_0})$ is dominated by $\exp\bigl(2 C \ell^{\alpha_0} \log(2\ell + 1)\bigr)$ we can choose $C_4 > 0$ large enough to deduce that $|P_4(\ell)| \leq \exp(C_4 \ell^{\alpha_0} \log \ell)$ for every $n$ and every $1 < \ell < r_n/2$.
\end{proof}

\section{The generated homeomorphism group is Thompson's group $V$}\label{sec:generating-V}

The goal of this section is to show that the group $W_n$ defined in~\cref{not:T_n} is isomorphic to Thompson's group $V$.
The group $W_n$ acts on the disjoint union of $|Y_n|=4\cdot 2^n$ Cantor sets
\begin{equation}\label{eq: partition of Cantor set}
  \CC_{Y_n}=\bigsqcup_{(x,i)\in Y_n=L_n\times\{1,\dots,4\}} \CC_{(x,i)},
\end{equation}
which we regard as a clopen partition of the (standard) Cantor set~$\CC$. Thompson's groups~V, F, and~T act on~$\CC$. 

\begin{remark}
By definition, $\Sym(4)_{\eta_n}$ is contained in $W_n$. Since $\mathcal{G}_n$ acts transitively on each $L_n$, we obtain immediately by conjugation with elements in $\mathcal{G}_n$ that $\Sym(4)_x\le W_n$ for every $x\in L_n$. We will use this repeatedly in the sequel. 
\end{remark}

\begin{lemma}\label{lem: Thompson F in H}
The group $W_n$ contains $\RiSt_F\bigl(\CC_{(\rho_n,1)}\cup \dots\cup  \CC_{(\rho_n,4)}\bigr)$.
\end{lemma}
\begin{proof}
The generator $X_0^{(n)}\in \RiSt_F\bigl(\CC_{(\rho_n,2)}\cup \CC_{(\rho_n,3)}\bigr)\cap T_n$ corresponds to
the standard generator~$X_0$ of $F$ under the identification of $\CC_{(\rho_n,2)}$ as the left half interval and
$\CC_{(\rho_n,3)}$ as the right half interval of the standard Cantor set. 
The generator $X_1^{(n)}\in \RiSt_F\bigl(\CC_{(\rho_n,4)}\bigr)\cap T_n$ corresponds to $X_0$ under the identification 
of $\CC_{(\rho_n, 4)}$ with the standard Cantor set. The conjugate of $X_1^{(n)}$ by $(3,4)\in\Sym(4)_{\rho_n}$ is supported on $\CC_{(\rho_n,3)}$, and corresponds to the second standard generator~$X_1$ of~$F$ under the identification of $\RiSt_F\bigl(\CC_{(\rho_n,2)}\cup \CC_{(\rho_n,3)}\bigr)$ with $F$ acting on $\CC_{(\rho_n,2)}\cup \CC_{(\rho_n,3)}$.
Thus, $\RiSt_F\bigl(\CC_{(\rho_n,2)}\cup \CC_{(\rho_n,3)}\bigr)=\langle X_0^{(n)}, (3,4)X_1^{(n)}(3,4)\rangle\le W_n$. 
By conjugation with elements in $\Sym(4)_{\rho_n}$ we obtain two more copies of $F$ inside $W_n$, namely 
\[ \RiSt_F\bigl(\CC_{(\rho_n,1)}\cup \CC_{(\rho_n,2)}\bigr)\le W_n~\text{ and }~\RiSt_F\bigl(\CC_{(\rho_n,3)}\cup \CC_{(\rho_n,4)}\bigr)\le W_n.\]
We regard the three overlapping sets $\CC_{(\rho_n,1)}\cup \CC_{(\rho_n,2)}$, $\CC_{(\rho_n,2)}\cup \CC_{(\rho_n,3)}$ and  $\CC_{(\rho_n,3)}\cup \CC_{(\rho_n,4)}$ as the intervals $[0,\frac{1}{2}]$, $[\frac{1}{4},\frac{3}{4}]$ and $[\frac{1}{2},1]$ in the standard Cantor set.
The standard generator $X_1$ of~$F$ is supported on $[\frac{1}{2},1]$.
In particular, it is an element of $\RiSt_F\bigl(\CC_{(\rho_n,3)}\cup \CC_{(\rho_n,4)}\bigr)$. 
The standard generator $X_0$ can be written as a composition of elements supported on $[0,\frac{1}{2}]$, $[\frac{1}{4},\frac{3}{4}]$ and $[\frac{1}{2},1]$ as indicated in~\cref{figure: Thompson generator}. Therefore, 
$\RiSt_F\bigl(\CC_{(\rho_n,1)}\cup \dots\cup  \CC_{(\rho_n,4)}\bigr)$ is generated by the three $F$-copies $\RiSt_F\bigl(\CC_{(\rho_n,1)}\cup \CC_{(\rho_n,2)}\bigr)$, $\RiSt_F\bigl(\CC_{(\rho_n,2)}\cup\CC_{(\rho_n,3)}\bigr)$ and $\RiSt_F\bigl(\CC_{(\rho_n,3)}\cup \CC_{(\rho_n,4)}\bigr)$, which completes the proof. \qedhere

\begin{figure}
\begin{tikzpicture}[
    scale=0.8,
    x=11cm,
    y=1.15cm,
    >={Latex[length=2mm]},
    interval/.style={line width=2pt, line cap=butt},
    pieceA/.style={interval, black},
    pieceB/.style={interval, gray!70},
    pieceC/.style={interval, gray!45},
    pieceD/.style={interval, gray!20},
    tick/.style={line width=0.4pt},
    every node/.style={font=\small},
    smallarrow/.style={
    black,
    line width=0.65pt,
    -{Stealth[length=1.4mm,width=0.9mm]}}
]

\draw[pieceA] (0,0) -- ({1/8},0);
\draw[pieceB] ({1/8},0) -- ({1/4},0);
\draw[pieceC] ({1/4},0) -- ({1/2},0);
\draw[pieceD] ({1/2},0) -- (1,0);

\draw[pieceA] (0,-1) -- ({1/4},-1);
\draw[pieceB] ({1/4},-1) -- ({3/8},-1);
\draw[pieceC] ({3/8},-1) -- ({1/2},-1);
\draw[pieceD] ({1/2},-1) -- (1,-1);

\draw[pieceA] (0,-2) -- ({1/4},-2);
\draw[pieceB] ({1/4},-2) -- ({1/2},-2);
\draw[pieceC] ({1/2},-2) -- ({5/8},-2);
\draw[pieceD] ({5/8},-2) -- (1,-2);

\draw[pieceA] (0,-3) -- ({1/4},-3);
\draw[pieceB] ({1/4},-3) -- ({1/2},-3);
\draw[pieceC] ({1/2},-3) -- ({3/4},-3);
\draw[pieceD] ({3/4},-3) -- (1,-3);

\draw[smallarrow] ({8.5/8},-0.15) -- ({8.5/8},-0.85)
    node[midway,right] {$X_0^{[0,\frac{1}{2}]}$};

\draw[smallarrow] ({8.5/8},-1.15) -- ({8.5/8},-1.85)
    node[midway,right] {$X_0^{[\frac{1}{4},\frac{3}{4}]}$};

\draw[smallarrow] ({8.5/8},-2.15) -- ({8.5/8},-2.85)
    node[midway,right] {$X_0^{[\frac{1}{2},1]}$};

\node[above] at ({1/16},0) {$A$};
\node[above] at ({3/16},0) {$B$};
\node[above] at ({3/8},0) {$C$};
\node[above] at ({3/4},0) {$D$};

\node[below] at ({1/8},-3.35) {$A$};
\node[below] at ({3/8},-3.35) {$B$};
\node[below] at ({5/8},-3.35) {$C$};
\node[below] at ({7/8},-3.35) {$D$};

\foreach \x/\label in {
    0/0,
    {1/8}/{\frac18},
    {1/4}/{\frac14},
    {1/2}/{\frac12},
    1/1
} {
    \draw[tick] (\x,-0.08) -- (\x,-0.18);
    \node[below] at (\x,-0.18) {$\label$};
}

\foreach \x/\label in {
    0/0,
    {1/4}/{\frac14},
    {3/8}/{\frac38},
    {1/2}/{\frac12},
    1/1
} {
    \draw[tick] (\x,-1.08) -- (\x,-1.18);
    \node[below] at (\x,-1.18) {$\label$};
}

\foreach \x/\label in {
    0/0,
    {1/4}/{\frac14},
    {1/2}/{\frac12},
    {5/8}/{\frac58},
    1/1
} {
    \draw[tick] (\x,-2.08) -- (\x,-2.18);
    \node[below] at (\x,-2.18) {$\label$};
}

\foreach \x/\label in {
    0/0,
    {1/4}/{\frac14},
    {1/2}/{\frac12},
    {3/4}/{\frac34},
    1/1
} {
    \draw[tick] (\x,-3.08) -- (\x,-3.18);
    \node[below] at (\x,-3.18) {$\label$};
}

\end{tikzpicture}

\caption{The factorization $X_0=X_0^{[\frac{1}{2},1]}X_0^{[\frac{1}{4},\frac{3}{4}]}X_0^{[0,\frac{1}{2}]}$ of the standard generator $X_0$ of Thompson's group~F. The four colors track the images of the four initial pieces
$A=[0,\frac18]$, $B=[\frac18,\frac14]$, $C=[\frac14,\frac12]$, and $D=[\frac12,1]$.}
\label{figure: Thompson generator}
\end{figure}
\end{proof}

\begin{lemma}\label{lem: copy of V in H}
The group $W_n$ contains $\RiSt_V\bigl(\CC_{(\rho_n,1)}\cup \dots\cup  \CC_{(\rho_n,4)}\bigr)$.
\end{lemma}

\begin{proof}
The permutation $(1,2,3,4)\in\Sym(4)_{\rho_n}$, which cyclically permutes the pieces $\CC_{(\rho_n,i)}$, $i\in\{1,\dots, 4\}$, and 
the group $\RiSt_F\bigl(\CC_{(\rho_n,1)}\cup \dots\cup  \CC_{(\rho_n,4)}\bigr)$ generate the group $\RiSt_T\bigl(\CC_{(\rho_n,1)}\cup \dots\cup  \CC_{(\rho_n,4)}\bigr)$.
Indeed, this follows from the maximality of $F$ in $T$~\cite[Section~1]{BelkBleakQuickSkipper2024} and the fact that $(1,2,3,4)$, being fixed-point free, does not lie in $F$.
Likewise, we deduce from the maximality of $T$ in $V$~\cite[Theorem~3.7]{BelkBleakQuickSkipper2024}
that the transposition $(3,4)\in\Sym(4)_{\rho_n}$ and the group
$\RiSt_T\bigl(\CC_{(\rho_n,1)}\cup \dots\cup  \CC_{(\rho_n,4)}\bigr)$ generate $\RiSt_V\bigl(\CC_{(\rho_n,1)}\cup \dots\cup  \CC_{(\rho_n,4)}\bigr)$.
\end{proof}

\begin{theorem}\label{thm: the W group is V}
The group $W_n$ is isomorphic to Thompson's group~V.
\end{theorem}
\begin{proof}
Recall that the transposition $\tau_n \in T_n$ interchanges the two copies $\CC_{(\rho_n,1)}$ and $\CC_{(\theta_n,1)}$, where $\theta_n = 01^{n-1}$ and $\rho_n = 1^n$, and fixes every other copy.
By conjugating the inclusion from~\cref{lem: copy of V in H} with $\tau_n$ we obtain
\[
\RiSt_V\bigl(\CC_{(\theta_n, 1)}\cup \CC_{(\rho_n, 2)}\cup \CC_{(\rho_n, 3)}\cup \CC_{(\rho_n, 4)}\bigr)\le W_n.
\]
Since $\mathcal{G}_n \leq W_n$ acts transitively on $L_n$, there is an element $g \in \mathcal{G}_n$ with $\rho_n g = \theta_n$.
Thus we can conjugate the inclusion from~\cref{lem: copy of V in H} with $g$ to deduce
\[
\RiSt_V\bigl(\CC_{(\theta_n, 1)}\cup \CC_{(\theta_n, 2)}\cup \CC_{(\theta_n, 3)}\cup \CC_{(\theta_n, 4)}\bigr)\le W_n.
\]
By~\cref{lem:rigid-stabilizer-union} we thus obtain $\RiSt_V(U)\le W_n$ for 
\[ U=\bigcup_{i=1}^4\CC_{(\rho_n,i)}\cup\CC_{(\theta_n, i)}.\]
The set $\mathcal{U}=\{Ug\mid g\in \mathcal{G}_n\}$ is a clopen covering of~$\CC_{Y_n}$.
The conjugation of $\RiSt_V(U)$ by $h\in \mathcal{G}_n$ is $\RiSt_V(Uh)$. So~\cref{lem:rigid-stabilizer-union} implies the statement provided the nerve of~$\mathcal{U}$ is connected. 
For the latter, it suffices to prove that the graph $C$ on the vertex set $L_n$ whose edges are $\mathcal{G}_n$-translates of the edge $(\theta_n, \rho_n)$ is connected.
To this end, recall the metric $\delta$ on $L_n$ from~\cref{sec:Grigorchuk-background}, given by $\delta(x,y) = e^{-\ell(x,y)}$, where $\ell(x,y)$ denotes the length of the maximal common prefix of $x$ and $y$.
By~\cref{lem:pair-transitivity}, the diagonal $\mathcal{G}_n$-action on $L_n \times L_n$ is transitive on pairs of points of the same $\delta$-distance.
Because of $\delta(0x, \rho_n) = \delta(\theta_n,\rho_n)=1$ for every $x\in \{0,1\}^{n-1}$, the right-most leaf $\rho_n$ of the finite binary tree $L_n$ has an edge to every leaf in the left half of the tree. Repeating this argument for other vertices, we obtain that the graph is the complete bipartite graph $K_{2^{n-1}, 2^{n-1}}$. In particular, it is connected.
This proves $W_n = V_{Y_n}$.
Finally, by a result of Higman~\cite{Hig74}, we have $V_{Y_n} \cong V$, which completes the proof.
\end{proof}

\section{Non-uniform exponential growth of Thompson's group~$V$}

In this section we complete the proof that Thompson's group~$V$ has non-uniform exponential growth.
The fact that $V$ has exponential growth is well known and follows, for example, from the (uniform) exponential growth of its subgroup $F$~\cite{delaHarpe2002}.
It therefore remains to show that the infimal  exponential growth rate of $V$ satisfies $\omega(V)=1$. 

\begin{proof}[Proof of Theorem~\ref{thm:main}]
By~\cref{thm: the W group is V} there is an isomorphism between $W_n$ and Thompson's group~$V$.
Let $E_n\subset V$ be the image of the generating set $T_n\subset W_n$ (see~\cref{not:T_n}) under the isomorphism.
In particular, we have 
\[ \gamma_{W_n}^{T_n}(\ell)=\gamma_{V}^{E_n}(\ell)\]
for every $\ell>0$. 
As before, we write $r_n = 2^n-2$.
So by~\cref{lem:normal-form} (normal form for elements in~$W_n$) we obtain that
\[
\gamma_V^{E_n}(\ell)\le |P_1(\ell)|\cdot |P_2(\ell)|\cdot |P_3(\ell)|\cdot |P_4(\ell)|
\]
for every $\ell < r_n/2$.
By~\cref{lem:count-p1},~\cref{lem:count-p2}, \cref{lem:count-p3} and~\cref{lem:count-p4} there exist constants $C_i>0$ for $i\in\{1,2,3,4\}$ such that for all $n \ge 1$ and all $1 < \ell < r_n/3$ we have
\begin{align*}
    |P_1(\ell)|&\le \exp\bigl(C_1\ell^{\alpha_0}\bigr)\\
    |P_2(\ell)|&\le \exp\bigl(C_2\ell^{\alpha_0}\bigr)\\
    |P_3(\ell)|&\le \exp\bigl(C_3\ell^{\alpha_0}\log(\ell)\bigr)\\
    |P_4(\ell)|&\le \exp\bigl(C_4\ell^{\alpha_0}\log(\ell)\bigr),
 \end{align*}
where $\alpha_0$ is the constant from~\cref{not:alpha0}.
Let $C=4\cdot\max\{C_1,\dots,C_4\}$ and fix $\alpha\in (\alpha_0,1)$.
Then there exists $\ell_0 \in \N$ with
\[
\gamma_V^{E_n}(\ell)\le \exp\bigl(C \ell^\alpha\bigr)
\]
for $n \in \N$ and $\ell_0 \leq \ell < r_n/3$.
We conclude that
\begin{equation*}
 \omega(V)\le \inf_{n\in\N}\omega(V, E_n)
                    \le \inf_{n\in\N}\inf_{\ell\in\N} \bigl(\gamma_V^{E_n}(\ell)\bigr)^{1/\ell}\\
                    \le \inf_{n \in \N}\, \inf_{\ell_0 \leq \ell < r_n/3} \exp\bigl(C\ell^\alpha\bigr)^{1/\ell}=1.\qedhere
\end{equation*}
\end{proof}

\section{Generating sets for automorphism groups of free groups}\label{sec:aut}

The goal of this section is to construct generating sets for automorphism groups of free groups that yield the growth estimates stated in~\cref{thm:mainestimate}.
These generating sets are modeled on those constructed for Thompson's group~$V$ in the preceding sections.
For a finite set $Y$ we write $F(Y)$ to denote the free group with basis $\{u_y\mid y\in Y\}$.
A \emph{signed basis permutation} of $F(Y)$ is an automorphism of the form
\[
  u_y\mapsto u_{\sigma(y)}^{\epsilon_y},
  \qquad
  \sigma\in\Sym(Y),\quad \epsilon_y\in\{\pm1\}.
\]
The group of all signed basis permutations is called the
\emph{monomial subgroup} of $\Aut(F(Y))$ and will be denoted by $M(Y)$.
We have \[M(Y)\cong (\Z/2)^Y\rtimes \Sym(Y).\] 
For distinct $y,z\in Y$, let $\lambda_{y,z}\in\Aut(F(Y))$ be the \emph{Nielsen transvection}
\[
  u_y\mapsto u_yu_z,
  \qquad
  u_{y'}\mapsto u_{y'}
  \quad\text{for }y'\neq y.
\]
The following fact is well-known~\cite[Chapter I, Section 4]{LyndonSchupp01}.

\begin{theorem}\label{thm:nielsen-generating-aut-free}
For a finite set $Y$, the group $\Aut(F(Y))$ is generated by all signed basis permutations and all Nielsen transvections.
\end{theorem}

As in the preceding sections, we write
\[
Y_n=L_n\times\{1,2,3,4\}.
\]
Thus $F(Y_n)$ is a free group of rank $4\cdot 2^n$ and hence isomorphic to $F_{4\cdot 2^n}$.
We now define a finite subset
\[
\mathcal{A}_n \subseteq \Aut(F(Y_n)),
\]
which will be shown to generate $\Aut(F(Y_n))$.
Its construction is modeled on that of the generating set $T_n$ for Thompson's group~$V$ and is again organized into four types.

\begin{itemize}
\item\textbf{Type 1: The truncated Grigorchuk generators.}
For every generator $s_n\in S_n$ of $\mathcal{G}_n$, we denote -- by the same symbol -- the automorphism of $F(Y_n)$ given by
\[
  u_{(x,i)}\mapsto u_{(x s_n,i)}
\]
for every $(x,i)\in Y_n$. 

\item\textbf{Type 2: The monomial group over $\eta_n$.}
Let $M_{\eta_n}(Y_n)\leq \Aut(F(Y_n))$ be the group of all signed permutations of the four basis elements
\[
  u_{(\eta_n,1)},\ldots,u_{(\eta_n,4)}
\]
that fix all other basis elements.
We call $M_{\eta_n}(Y_n)$ the \emph{monomial group of the fiber over $\eta_n$}.
It is isomorphic to $(\Z/2\Z)^4\rtimes\Sym(4)$. 

\item\textbf{Type 3: The linking transposition.}
Let $\tau_n$ be the basis transposition swapping
$u_{(\rho_n,1)}$ and $u_{(\theta_n,1)}$ 
and fixing all other basis elements.

\item\textbf{Type 4: The Nielsen transvection.}
Let $\nu_n$ be the Nielsen transvection
\[
  \nu_n=\lambda_{(\rho_n,2),(\rho_n,3)}.
\]
\end{itemize}
\begin{definition}\label{def:Tn}
We define $
  \mathcal{A}_n\defeq S_n\cup M_{\eta_n}(Y_n)\cup\{\tau_n,\nu_n\}
  \subseteq \Aut(F(Y_n))$.
\end{definition}

\begin{lemma}\label{lem:full-signed-monomial}
The subgroup $\langle \mathcal{A}_n\rangle$ contains the monomial subgroup $M(Y_n)$. 
\end{lemma}
\begin{proof}
By spherical transitivity of $\mathcal{G}$, the group $\mathcal{G}_n$ acts transitively on $L_n$.
Hence, for every $x\in L_n$, the fiberwise monomial subgroup $M_x(Y_n)$ is a conjugate of $M_{\eta_n}(Y_n)$ by some word over $S_n$.
In particular, $\langle \mathcal{A}_n\rangle$ contains every signed basis transposition inside each fiber.
The graph $\Lambda_n$ on the vertex set $L_n$ whose edges are
\[
  \{\rho_n g,\theta_n g\},~~g\in\mathcal{G}_n
\]
is connected, which is shown in the proof of Theorem~5.4.
The conjugates of $\tau_n$ by elements of $\mathcal{G}_n$ are precisely the transpositions of the basis elements
\[
  u_{(\rho_n g,1)}
  \quad\text{and}\quad
  u_{(\theta_n g,1)}
\]
corresponding to the edges of $\Lambda_n$.
Together with the transpositions inside the fibers, these are the edge transpositions of a connected graph on the vertex set $Y_n$. 
Therefore they generate $\Sym(Y_n)$.
Since the inversions of the basis elements are contained in the conjugates of $M_{\eta_n}(Y_n)$, the statement follows.
\end{proof}

\begin{theorem}\label{thm:An-generates-Aut}
The set $\mathcal{A}_n$ generates $\Aut(F(Y_n))$. 
\end{theorem}
\begin{proof}
By~\cref{lem:full-signed-monomial}, the group $\langle \mathcal{A}_n\rangle$ contains all signed permutations of the basis.
It also contains the Nielsen transvection $\nu_n$.
Conjugating $\nu_n$ by the signed basis permutations on $Y_n$, we obtain every Nielsen transvection. 
Hence $\langle \mathcal{A}_n\rangle=\Aut(F(Y_n))$ by~\cref{thm:nielsen-generating-aut-free}.
\end{proof}

\section{Growth estimates for automorphism groups of free groups}

In this section, we estimate the growth of $\Aut(F(Y_n))$ with respect to the generating sets $\mathcal A_n$ constructed in the preceding section.
We continue to use the notation
\[
r_n\defeq 2^n-2
\]
introduced in~\cref{prop:BE-finite}. 

\begin{lemma}\label{lem:commuting-supports}
Let $n \geq 2$ and let $g,h\in S_n^\ast$ have length strictly less than $r_n/2$.
Then every  element of $M_{\eta_n}(Y_n)^h$ commutes with the $g$-conjugates $\tau_n^g$ and $\nu_n^g$.
\end{lemma}
\begin{proof}
The conjugate of $M_{\eta_n}(Y_n)$ by $h$ is supported on the basis elements in the fiber over $\eta_nh$.
On the other hand, $\nu_n^g$ is supported on the basis elements $(\rho_ng,2)$ and $(\rho_ng,3)$, while $\tau_n^g$ is supported on $(\rho_ng,1)$ and $(\theta_ng,1)$.
Thus it is enough to prove that $\eta_nh$ is distinct from both $\rho_ng$ and $\theta_ng$.
To see this, we recall that
\begin{equation}\label{eq:recall-distance}
1+r_n=2^n-1=d_{\Gamma_n}(\rho_n,\eta_n).
\end{equation}
If $\eta_nh=\rho_ng$, then
\[
d_{\Gamma_n}(\rho_n,\eta_n)
\leq d_{\Gamma_n}(\rho_n,\rho_ng)
+d_{\Gamma_n}(\rho_ng,\eta_nh)
+d_{\Gamma_n}(\eta_nh,\eta_n)
\leq \abs{g}_{S_n}+\abs{h}_{S_n}
< r_n,
\]
which contradicts~\eqref{eq:recall-distance}.
Similarly, if $\eta_nh=\theta_ng$, then
\[
\begin{aligned}
d_{\Gamma_n}(\rho_n,\eta_n)
&\leq d_{\Gamma_n}(\rho_n,\theta_n)
+d_{\Gamma_n}(\theta_n,\theta_ng)
+d_{\Gamma_n}(\theta_ng,\eta_nh)
+d_{\Gamma_n}(\eta_nh,\eta_n)\\
&\leq 1+\abs{g}_{S_n}+\abs{h}_{S_n}\\
&<1+r_n,
\end{aligned}
\]
which is again a contradiction.
Hence the corresponding supports are disjoint and the claim follows.
\end{proof}

\begin{lemma}\label{lem:normal-form-2}
Let $n \geq 2$, and let $w$ be a word over $\mathcal{A}_n\cup \mathcal{A}_n^{-1}$ of length $\ell < r_n/2$.
Then the element represented by $w$ can be written as
\[
p_1 \cdot p_2 \cdot p_3 \cdot p_4,
\]
where $p_1$ is represented by a word of the form $g_1\cdots g_{\ell}$ over $S_n\cup\{1\}$, and the remaining factors have the form
\begin{align*}
  p_2
  &=
  (a_1^{(2)})^{g_1g_2\cdots g_{\ell}} \cdot
  (a_2^{(2)})^{g_2\cdots g_{\ell}}
  \cdot
  \ldots
  \cdot
  (a_{\ell}^{(2)})^{g_{\ell}},\\
  p_3
  &=
  (a_1^{(3)})^{g_1g_2\cdots g_{\ell}} \cdot
  (a_2^{(3)})^{g_2\cdots g_{\ell}}
  \cdot
  \ldots
  \cdot
  (a_{\ell}^{(3)})^{g_{\ell}},\\
  p_4
  &=
  (a_1^{(4)})^{g_1g_2\cdots g_{\ell}} \cdot
  (a_2^{(4)})^{g_2\cdots g_{\ell}}
  \cdot
  \ldots
  \cdot
  (a_{\ell}^{(4)})^{g_{\ell}},
\end{align*}
where for each $i$ we have
\[
  g_i\in S_n\cup\{1\},
  \qquad
  a_i^{(2)}\in M_{\eta_n}(Y_n),
  \qquad
  a_i^{(3)}\in\{\tau_n,1\},
  \qquad
  a_i^{(4)}\in\{\nu_n,\nu_n^{-1},1\}.
\]
\end{lemma}

\begin{proof}
We write $w$ as a product $t_1 \ldots t_\ell$ of letters $t_i \in \mathcal{A}_n\cup\mathcal{A}_n^{-1}$ and write each letter in the form
\[
  t_i=a_i^{(2)}a_i^{(3)}a_i^{(4)}g_i,
\]
where $g_i\in S_n\cup\{1\}$, $a_i^{(2)}\in M_{\eta_n}(Y_n)$, $a_i^{(3)}\in\{\tau_n,1\}$, and $a_i^{(4)}\in\{\nu_n,\nu_n^{-1},1\}$, with at most one of $a_i^{(2)},a_i^{(3)},a_i^{(4)}$ non-trivial.
This is possible because $M_{\eta_n}(Y_n)$ is a group, the elements of $S_n$ and $\tau_n$ are involutions, and $\mathcal{A}_n \cup \mathcal{A}_n^{-1}$ only adds $\nu_n^{-1}$.
Pushing all $g_i$'s to the left gives
\[
t_1\cdots t_\ell
=
g_1\cdots g_\ell
\cdot
\prod_{i=1}^{\ell}
\bigl(a_i^{(2)}a_i^{(3)}a_i^{(4)}\bigr)^{g_i\cdots g_\ell},
\]
where $x^g = g^{-1}xg$. For $j\in\{2,3,4\}$, call the elements
\[
  (a_i^{(j)})^{g_i\cdots g_\ell}
  \qquad (1\leq i\leq \ell)
\]
the conjugated factors of type $j$. Thus ``type'' refers to these conjugates appearing after the rewriting, not to different generators of type $3$ or type $4$.

For each $i$, the suffix $g_i\cdots g_\ell$ has $S_n$-length at most $\ell < r_n/2$.
Hence \cref{lem:commuting-supports} implies that every conjugated factor of type $2$ commutes with every conjugated factor of type $3$ and with every conjugated factor of type $4$.
A conjugated factor of type $3$ is supported on sheet $1$, whereas a conjugated factor of type $4$ is supported on sheets $2$ and $3$; hence these two factors commute as well.
We may therefore move all conjugated factors of type $2$ before all conjugated factors of type $3$, and all conjugated factors of type $3$ before all conjugated factors of type $4$, preserving the order inside each type.
This gives exactly the three displayed products $p_2,p_3,p_4$, and hence the claimed normal form.
\end{proof}

We are now ready to prove the growth estimate for automorphism groups of free groups stated in the introduction.

\begin{proof}[Proof of~\cref{thm:mainestimate}]
Recall the constant $\alpha_0$ from~\cref{not:alpha0} and fix $\alpha\in(\alpha_0,1)$.
Let $P_i(\ell)$ denote the set of all possible factors $p_i$ appearing in the normal form of \cref{lem:normal-form-2} for words of length at most $\ell\leq r_n/3$.
The first factor satisfies
\[
  \abs{P_1(\ell)}\leq \exp(C_1\ell^{\alpha_0})
\]
by \cref{thm:G-subexponential}, after passing to the finite quotient $\mathcal{G}_n$.
Let $C_0>0$ be a constant such that the estimates from \cref{prop:BE-finite} hold with $C_0$ for both $\rho_n$ and $\eta_n$.
To estimate the number of possible factors $p_2$, we use the notation of~\cref{lem:normal-form-2} and set
\[
u\defeq g_1\cdots g_\ell.
\]
Each conjugate
\[
  (a_i^{(2)})^{g_i\cdots g_\ell}
\]
is supported on the fiber over the point $\eta_n g_i\cdots g_\ell$.
Hence $p_2$ is supported on the inverted orbit $\OO_{\eta_n}(u)$.
By~\cref{prop:BE-finite}, there are at most
$\exp(C_0\ell^{\alpha_0})$ inverted orbits
$\OO_{\eta_n}(u)$ arising from words $u$ of length at most $\ell$, and each such orbit has cardinality at most
$C_0\ell^{\alpha_0}$.
At each point of an orbit the value of $p_2$ lies in a copy of the group $M_{\eta_n}(Y_n)$, which is of order $2^4\cdot 4!$.
Therefore
\[
  \abs{P_2(\ell)}
  \leq
  \exp(C_0\ell^{\alpha_0})\cdot (2^4\cdot 4!)^{C_0\ell^{\alpha_0}}
  \leq
  \exp(C_2\ell^{\alpha_0})
\]
for a suitable constant~$C_2>0$.
The factor $p_3$ of a word of length~$\ell$ is a permutation of the sheet-one basis elements.
Each non-trivial conjugate
\[
  (a_i^{(3)})^{g_i\cdots g_\ell}
\]
is the transposition of $u_{(\rho_n g_i\cdots g_\ell,1)}$ and $u_{(\theta_n g_i\cdots g_\ell,1)}$. 
Thus the support of $p_3$ is contained in
\[
  \{\rho_n g_i\cdots g_\ell,\ \theta_n g_i\cdots g_\ell\mid 1\leq i\leq \ell\}\times\{1\}.
\]
We now show that this support is contained in a single inverted orbit based at $\rho_n$.
To this end, consider the word
\[
\widetilde u
\defeq
a_ng_1a_n\,a_ng_2a_n\cdots
a_ng_{\ell-1}a_n\,a_ng_\ell.
\]
This is a word over $S_n\cup\{1\}$ of length at most $3\ell\le r_n$.
Since $\theta_n=\rho_na_n$ and $a_n^2=1$, the suffixes of $\widetilde u$ starting at $g_i$ give the points $\rho_n g_i\cdots g_\ell$, while the suffixes starting at the $a_n$ immediately before $g_i$ give the points $\theta_n g_i\cdots g_\ell$.
Hence the support of $p_3$ is contained in $\OO_{\rho_n}(\widetilde u)\times\{1\}$.
Another application of~\cref{prop:BE-finite} tells us that there are at most $\exp\bigl(C_0(3\ell)^{\alpha_0}\bigr)$
possible inverted orbits $\OO_{\rho_n}(\widetilde u)$, each of cardinality at most
$C_0(3\ell)^{\alpha_0}$.
For a fixed orbit $O$, the factor $p_3$ is a permutation of the
sheet-one basis elements over $O$.
Thus there are at most $\lvert O\rvert!$ possibilities for $p_3$.
Hence, setting
\[
m_\ell\defeq\left\lceil C_0(3\ell)^{\alpha_0}\right\rceil,
\]
we obtain
\[
\abs{P_3(\ell)}
\leq
\exp\bigl(C_0(3\ell)^{\alpha_0}\bigr)m_\ell!
\leq
\exp\bigl(C_3\ell^{\alpha_0}\log\ell\bigr)
\]
for a suitable constant $C_3>0$ and every $\ell \geq 2$.
Finally, we consider the factor $p_4$ of a word of length~$\ell$. Setting $u=g_1\cdots g_\ell$, the factor $p_4$ is supported on the inverted orbit $\OO_{\rho_n}(u)$.
Again, the number of possible such orbits is at most $\exp(C_0\ell^{\alpha_0})$, and each has cardinality at most $C_0\ell^{\alpha_0}$.
At a point $x$ of this orbit, the possible contribution is a power of the transvection
\[
  u_{(x,2)}\mapsto u_{(x,2)}u_{(x,3)}
\]
with exponent between $-\ell$ and $\ell$.
So there are at most $2\ell+1$ choices at each point and we obtain
\[
\abs{P_4(\ell)}
\leq
\exp(C_0\ell^{\alpha_0})\cdot (2\ell+1)^{C_0\ell^{\alpha_0}}
\leq
\exp(C_4\ell^{\alpha_0} \log \ell)
\]
for a suitable constant $C_4>0$.
Since
\[
\ell^{\alpha_0}\log\ell\leq C_\alpha\ell^\alpha
\]
for each $\ell\geq2$ and a suitable constant $C_\alpha>0$, the preceding estimates yield a constant $C>0$, independent of $n$, such that
\[
\gamma_{\Aut(F(Y_n))}^{\mathcal A_n}(\ell)
\leq
\prod_{i=1}^4\abs{P_i(\ell)}
\leq
\exp\bigl(C\ell^{\alpha}\bigr)
\]
for every $n$ and every $\ell \leq r_n/3$.
Since $F(Y_n)\cong F_{2^{n+2}}$ and $r_n=2^n-2$, this proves the first assertion of~\cref{thm:mainestimate}.

To prove the second assertion we fix $\beta\in(0,1-\alpha_0)$ and $\alpha\in(\alpha_0,1-\beta)$.
Applying the first assertion to this $\alpha$, we obtain a constant $C>0$ such that
\[
\gamma_{\Aut(F(Y_n))}^{\mathcal A_n}(\ell)\leq\exp\bigl(C\ell^{\alpha}\bigr)
\]
for every $n$ and every $0\leq\ell\leq r_n/3$.
By setting $\ell_n\defeq\bigl\lfloor\frac{r_n}{3}\bigr\rfloor$
we therefore obtain
\[
\omega\bigl(\Aut(F_{2^{n+2}})\bigr)
\leq
\omega\bigl(\Aut(F(Y_n)),\mathcal A_n\bigr)
\leq
\gamma_{\Aut(F(Y_n))}^{\mathcal A_n}(\ell_n)^{1/\ell_n}
\leq
\exp\bigl(C\ell_n^{\alpha-1}\bigr)
\]
for all sufficiently large $n$.
Since $\ell_n\geq 2^{n-2}$ for sufficiently large $n$ we deduce that
\[
C\ell_n^{\alpha-1}
\leq
C\,2^{-(1-\alpha)(n-2)}
\leq
2^{-\beta(n+2)}
\]
for sufficiently large $n$, which proves the second assertion.
\end{proof}

\section{A common quotient of the groups $\Aut(F_n)$}
\label{sec:acylhyp}

The following statement is a direct consequence of the fact that the automorphism group of a free group of rank at least~2 is acylindrically hyperbolic, due to Genevois--Horbez~\cite[Corollary~1.2]{GenevoisHorbez2020}, and the common quotient theorem of Minasyan--Osin~\cite[Theorem~1.1]{MinasyanOsin2018}. The latter says that every countable family of countable acylindrically hyperbolic groups has a common finitely generated acylindrically hyperbolic quotient.

\begin{theorem}\label{thm: common quotient}
The automorphism groups of free groups $\Aut(F_n)$, $n\ge 2$, have a common finitely generated acylindrically hyperbolic quotient. 
\end{theorem}

\begin{proof}[Proof of~\cref{thm:acylhyp}]
By~\cref{thm: common quotient} there is a finitely generated acylindrically hyperbolic group $Q$ and epimorphisms
\[
  q_n\colon \Aut(F_{2^{n+2}})\twoheadrightarrow Q
\]
for every $n\in\N$.
The group $Q$ is a Kazhdan group since it is a quotient of Kazhdan groups by the result of Kaluba-Kielak-Nowak~\cite[Theorem~1.1]{KalubaKielakNowak2021} or the earlier result of Kaluba--Nowak--Ozawa~\cite{KalubaNowakOzawa2018}. See also~\cref{rem: property T} below.
For every $n\in\N$, let
\[
U_n\defeq q_n(\mathcal A_n).
\]
By~\cref{thm:An-generates-Aut}, the set $U_n$ is a finite generating
set of $Q$.
Since taking quotients cannot increase the growth function with respect to the image
generating set,~\cref{thm:mainestimate} provides us, for any fixed $\alpha\in(\alpha_0,1)$, with a constant $C>0$ with
\[
\gamma_Q^{U_n}(\ell)
\leq
\gamma_{\Aut(F_{2^{n+2}})}^{\mathcal A_n}(\ell)
\leq
\exp\bigl(C\ell^{\alpha}\bigr)
\]
for every $0\leq\ell\leq(r_n)/3$.
By setting $\ell_n \defeq \left\lfloor\frac{r_n}{3}\right\rfloor$ we obtain
\[
1
\leq \omega(Q)
\leq \omega(Q,U_n)
\leq \gamma_Q^{U_n}(\ell_n)^{1/\ell_n}
\leq \exp\bigl(C\ell_n^{\alpha-1}\bigr)
\]
for all sufficiently large $n$.
Since $\ell_n\to\infty$ and $\alpha<1$, the right-hand side converges to $1$, which shows that $\omega(Q)=1$.
On the other hand, it is well-known that $Q$ has exponential growth since an acylindrically hyperbolic group contains a non-abelian free group (coming from independent loxodromic elements)~\cite{Osin16}. 
So $Q$ has non-uniform exponential growth.
\end{proof}

\begin{remark}\label{rem: property T}
To obtain the Kazhdan property for~$Q$ one does not have to appeal to results about the Kazhdan property for automorphism groups of free groups. One could simply take the common quotient of $Q$ as above and a Kazhdan hyperbolic group. 
\end{remark}

\section{Generating sets for elementary linear groups}
\label{sec:el-generators}

In this section we construct the generating sets for elementary linear groups that give rise to~\cref{thm:el}.
The construction of these generating sets is inspired by the generating sets $\mathcal{A}_n$ of $\Aut(F_{2^{n+2}})$ from~\cref{sec:aut}, with elementary matrices taking the place of Nielsen transvections.
To make this precise, we fix a finitely generated associative unital ring $R$ and a finite set $Y$.

\begin{notation}
For distinct $y,z\in Y$ and $r\in R$, let
\[
e_{y,z}(r)\defeq I+rE_{y,z}\in\GL_Y(R)
\]
be the corresponding elementary matrix.
We write
\[
\EL_Y(R)
\defeq
\bigl\langle e_{y,z}(r)
\mid y,z\in Y,\ y\neq z,\ r\in R\bigr\rangle.
\]
\end{notation}

We will use the relations
\begin{align}
e_{y,z}(r)e_{y,z}(s)&=e_{y,z}(r+s),
\label{eq:elementary-addition}\\
[e_{y,z}(r),e_{z,v}(s)]&=e_{y,v}(rs)
\label{eq:elementary-commutator}
\end{align}
for pairwise distinct $y,z,v\in Y$, where $[g,h]=ghg^{-1}h^{-1}$.
Since transposition matrices do not generally lie in $\EL_Y(R)$ we will work with signed transposition matrices instead.
For distinct $y,z\in Y$, such a signed transposition matrix $w_{y,z}\in\GL_Y(R)$ is given by the block
\[
\begin{pmatrix}
0&1\\
-1&0
\end{pmatrix}
\]
on the coordinates indexed by $y$ and $z$, while fixing all other coordinates.
A direct computation shows that
\[
w_{y,z}
=
e_{y,z}(1)e_{z,y}(-1)e_{y,z}(1),
\]
so that $w_{y,z}\in\EL_Y(R)$.
More generally, we will work with \emph{signed monomial matrices}, i.e.\ matrices $M$ that map each standard basis vector $b_y$ to $\pm b_{\sigma(y)}$ for some permutation $\sigma\in\Sym(Y)$, which we call the \emph{underlying permutation} of $M$.

\begin{definition}\label{def:elementary-monomial}
For a finite set $Y$ with $|Y|\geq2$, we define the
\emph{elementary monomial subgroup}
\[
\mathcal M_R(Y)
\defeq
\bigl\langle w_{y,z}\mid y,z\in Y,\ y\neq z\bigr\rangle
\leq\EL_Y(R).
\]
\end{definition}

\begin{lemma}\label{lem:elementary-monomial}
The following statements hold for every finite set $Y$ with $|Y|\geq 2$.
\begin{enumerate}
\item The cardinality of $\mathcal M_R(Y)$ satisfies
\[
|\mathcal M_R(Y)|\leq 2^{|Y|-1}|Y|!.
\]
\item For every $\sigma\in\Sym(Y)$, the group $\mathcal M_R(Y)$ contains a signed monomial matrix with underlying permutation $\sigma$.
Moreover, if $\sigma$ is even, then the permutation matrix associated with $\sigma$ itself belongs to $\mathcal M_R(Y)$.
\item If $\Lambda$ is a connected graph with vertex set $Y$, then
\[
\mathcal M_R(Y)
=
\bigl\langle w_{y,z}\mid\{y,z\}\in E(\Lambda)\bigr\rangle.
\]
\item Given two pairs $(y,z)$ and $(y',z')$ of distinct elements of $Y$, there is an element $q\in\mathcal M_R(Y)$ with
\[
q^{-1}e_{y,z}(r)q=e_{y',z'}(\varepsilon r)
\]
for every $r\in R$ and some $\varepsilon\in\{\pm1\}$.
\end{enumerate}
\end{lemma}
\begin{proof}
Let $D\leq\GL_Y(\Z)$ denote the group of signed monomial matrices of determinant one over $\Z$.
We first show that the matrices $w_{y,z}$ generate $D$ over $\Z$.
Each $w_{y,z}$ lies in $D$, so the generated group is contained in $D$.
Conversely, the underlying permutations of the $w_{y,z}$ contain all transpositions, while $w_{y,z}^2$ changes the signs of precisely the two coordinates $y$ and $z$.
These squares generate all diagonal sign changes involving an even number of minus signs.
Given $m\in D$ with underlying permutation $\sigma$, a product $w$ of matrices $w_{y,z}$ realizing $\sigma$ has the property that $w^{-1}m$ is a diagonal sign change of determinant one, so $m$ lies in the generated group.
Since $D$ has order $2^{|Y|-1}|Y|!$ and $\mathcal M_R(Y)$ is the image of $D$ under the homomorphism induced by the canonical unital ring homomorphism $\Z\to R$, the first assertion follows.

For the second assertion, let $\sigma\in\Sym(Y)$.
Since a sign change at an arbitrary $y\in Y$ changes the sign of the determinant, the group $D$ contains a signed monomial matrix with underlying permutation $\sigma$.
If $\sigma$ is even, its permutation matrix itself lies in $D$.
Passing to the image in $\mathcal M_R(Y)$ proves the second assertion.

If $\Lambda$ is connected, its edge transpositions generate $\Sym(Y)$.
Moreover, for any two $y,z\in Y$, multiplying the squares of the edge matrices along a path from $y$ to $z$ yields the sign change at precisely $y$ and $z$.
As every sign change at an even number of coordinates is a product of sign changes at two coordinates, the squares of the edge matrices generate all such sign changes.
Thus the argument of the first paragraph applies to the edge matrices of $\Lambda$.

For the fourth assertion, choose $\sigma\in\Sym(Y)$ mapping $(y',z')$ to $(y,z)$.
By the second assertion there is a signed monomial matrix $q\in\mathcal M_R(Y)$ with underlying permutation $\sigma$, say $qb_x=\varepsilon_xb_{\sigma(x)}$ with $\varepsilon_x\in\{\pm1\}$.
A direct computation gives
\[
q^{-1}e_{y,z}(r)q=e_{y',z'}(\varepsilon_{y'}\varepsilon_{z'}r)
\]
for every $r\in R$.
\end{proof}

Let us now fix a finite generating set $A$ of $R$ that satisfies $A=-A$ and $1\in A$.
As in the preceding sections, we consider the set
\[
Y_n=L_n\times\{1,2,3,4\}.
\]
We now define a finite subset
\[
\mathcal E_n(R)\subseteq\EL_{Y_n}(R),
\]
which will be shown to generate $\EL_{Y_n}(R)$.
As with the generating set $\mathcal A_n$ of $\Aut(F(Y_n))$, its construction is modeled on that of the generating set $T_n$ of Thompson's group $V$ and is organized into four types.

\smallskip

\begin{itemize}
\item\textbf{Type 1: The truncated Grigorchuk generators.}
For every $s_n\in S_n$, we denote by the same symbol the permutation matrix associated with the permutation
\[
(x,i)\longmapsto(xs_n,i)
\]
of $Y_n$.
This permutation is the disjoint union of four copies of the permutation induced by $s_n$ on $L_n$ and is therefore even.
It follows from~\cref{lem:elementary-monomial} that its permutation matrix belongs to $\EL_{Y_n}(R)$.

\item\textbf{Type 2: The elementary monomial group of the fiber over $\eta_n$.}
Let
\[
\mathcal M_{\eta_n}(Y_n;R)
\defeq
\mathcal M_R\bigl(\{\eta_n\}\times\{1,2,3,4\}\bigr)
\]
denote the subgroup of $\EL_{Y_n}(R)$ generated by the signed transpositions within the fiber $\{\eta_n\}\times\{1,2,3,4\}$ over $\eta_n$.

\item\textbf{Type 3: The linking signed transposition.}
We write
\[
w_n\defeq w_{(\rho_n,1),(\theta_n,1)}.
\]

\item\textbf{Type 4: The elementary matrices over $\rho_n$.}
For every $a\in A$, let
\[
\varepsilon_{n,a}
\defeq
e_{(\rho_n,2),(\rho_n,3)}(a).
\]
\end{itemize}

\begin{definition}\label{def:En}
We define the subset
\[
\mathcal E_n(R)
\defeq
S_n\cup\mathcal M_{\eta_n}(Y_n;R)
\cup\{w_n\}
\cup\{\varepsilon_{n,a}\mid a\in A\}.
\]
of $\EL_{Y_n}(R)$.
\end{definition}

\begin{lemma}\label{lem:full-elementary-monomial}
The subgroup $\mathcal M_R(Y_n)$ is contained in $\langle\mathcal E_n(R)\rangle$.
\end{lemma}
\begin{proof}
By spherical transitivity of $\mathcal G$, the group $\mathcal G_n = \langle S_n \rangle$ acts transitively on $L_n$.
Hence, for every $x\in L_n$, the elementary monomial group of the fiber over $x$ is a conjugate of $\mathcal M_{\eta_n}(Y_n;R)$ by an element of $\mathcal G_n$.
In particular, $\langle\mathcal E_n(R)\rangle$ contains every signed transposition $w_{y,z}$ for which $y$ and $z$ lie in the same fiber.
Let $\Lambda_n$ denote the graph with vertex set $L_n$ whose set of edges is given by
\[
\bigl\{\{\rho_ng,\theta_ng\}\mid g\in\mathcal G_n\bigr\}.
\]
From the proof of~\cref{thm: the W group is V} we know that $\Lambda_n$ is connected.
On the other hand, the conjugates of $w_n$ by the elements of $\mathcal G_n$ are precisely the signed transpositions
\[
w_{(\rho_ng,1),(\theta_ng,1)}.
\]
Together with the signed transpositions inside the fibers, these are the edge signed transpositions of a connected graph on the vertex set $Y_n$.
In this case the statement follows from~\cref{lem:elementary-monomial}.
\end{proof}

\begin{theorem}\label{thm:En-generates-EL}
The set $\mathcal E_n(R)$ generates $\EL_{Y_n}(R)$ for every $n \geq 1$.
\end{theorem}
\begin{proof}
From~\cref{lem:full-elementary-monomial} we know that the group $\mathcal M_R(Y_n)$ is contained in
\[
H_n\defeq\langle\mathcal E_n(R)\rangle.
\]
We can therefore conjugate the matrices $\varepsilon_{n,a}$ by elements of $\mathcal M_R(Y_n)$ to deduce that $e_{y,z}(a) \in H_n$ for all distinct $y,z\in Y_n$ and every $a\in A$.

Let us next consider the subset
\[
R_0
\defeq
\bigl\{r\in R\mid e_{y,z}(r)\in H_n
\text{ for all distinct }y,z\in Y_n\bigr\}
\]
of $R$.
By~\eqref{eq:elementary-addition}, the set $R_0$ is an additive subgroup of $R$.
If $r,s\in R_0$ and $y,z\in Y_n$ are distinct, we can choose $x\in Y_n\setminus\{y,z\}$.
Then~\eqref{eq:elementary-commutator} tells us that
\[
e_{y,z}(rs)
=[e_{y,x}(r),e_{x,z}(s)]
\in H_n.
\]
Since moreover $A\subseteq R_0$ and $1\in A$, it follows that $R_0$ is a unital subring of $R$.
As $A$ generates $R$ as a unital ring, we have $R_0=R$, which completes the proof.
\end{proof}

\section{Growth estimates for elementary linear groups}
\label{sec:el}

In this section, we estimate the growth of $\EL_{Y_n}(R)$ with respect to the generating sets $\mathcal E_n(R)$ constructed in the preceding section.
We continue to use the notation
\[
r_n = 2^n-2
\]
introduced in~\cref{prop:BE-finite}.

\begin{lemma}\label{lem:commuting-supports-EL}
Let $n\geq2$ and let $g,h\in S_n^\ast$ be words of length less than $r_n/2$.
Then every element of $\mathcal M_{\eta_n}(Y_n;R)^h$ commutes with the $g$-conjugates $w_n^g$ and $\varepsilon_{n,a}^g$ for every $a\in A$.
\end{lemma}
\begin{proof}
The conjugate of $\mathcal M_{\eta_n}(Y_n;R)$ by $h$ is supported on the four coordinates in the fiber over $\eta_nh$.
On the other hand, $w_n^g$ is supported on $(\rho_ng,1)$ and $(\theta_ng,1)$, while $\varepsilon_{n,a}^g$ is supported on $(\rho_ng,2)$ and $(\rho_ng,3)$.
In this case the same argument as in~\cref{lem:commuting-supports} shows that $\eta_nh$ is distinct from $\rho_ng$ and $\theta_ng$.
Hence the corresponding supports are disjoint and the claim follows.
\end{proof}
\begin{lemma}\label{lem:normal-form-EL}
Let $n\geq2$, and let $w$ be a word over $\mathcal E_n(R)\cup\mathcal E_n(R)^{-1}$ of length $\ell<r_n/2$.
Then the element represented by $w$ can be written as
\[
p_1\cdot p_2\cdot p_3\cdot p_4,
\]
where $p_1$ is represented by a word of the form $g_1\cdots g_\ell$ over $S_n\cup\{1\}$, and the remaining factors have the form
\begin{align*}
p_2&=
(a_1^{(2)})^{g_1\cdots g_\ell}\cdot
(a_2^{(2)})^{g_2\cdots g_\ell}\cdot
\ldots\cdot
(a_\ell^{(2)})^{g_\ell},\\
p_3&=
(a_1^{(3)})^{g_1\cdots g_\ell}\cdot
(a_2^{(3)})^{g_2\cdots g_\ell}\cdot
\ldots\cdot
(a_\ell^{(3)})^{g_\ell},\\
p_4&=
(a_1^{(4)})^{g_1\cdots g_\ell}\cdot
(a_2^{(4)})^{g_2\cdots g_\ell}\cdot
\ldots\cdot
(a_\ell^{(4)})^{g_\ell},
\end{align*}
where for each $i$ we have
\[
g_i\in S_n\cup\{1\},
\qquad
a_i^{(2)}\in\mathcal M_{\eta_n}(Y_n;R),
\qquad
a_i^{(3)}\in\{w_n,w_n^{-1},1\},
\]
and
\[
a_i^{(4)}\in\{\varepsilon_{n,a}\mid a\in A\}\cup\{1\}.
\]
\end{lemma}

\begin{proof}
We write $w$ as a product $t_1\ldots t_\ell$ of letters in $\mathcal E_n(R)\cup\mathcal E_n(R)^{-1}$ and write each letter in the form
\[
t_i=a_i^{(2)}a_i^{(3)}a_i^{(4)}g_i,
\]
with at most one of $a_i^{(2)},a_i^{(3)},a_i^{(4)}$ non-trivial.
This is possible because $\mathcal M_{\eta_n}(Y_n;R)$ is a group, the elements of $S_n$ are involutions, and
\[
\varepsilon_{n,a}^{-1}=\varepsilon_{n,-a}
\]
with $-a\in A$.
Pushing all $g_i$'s to the left gives
\[
t_1\cdots t_\ell
=
g_1\cdots g_\ell
\cdot
\prod_{i=1}^{\ell}
\bigl(a_i^{(2)}a_i^{(3)}a_i^{(4)}\bigr)^{g_i\cdots g_\ell}.
\]
For each $i$, the suffix $g_i\cdots g_\ell$ has $S_n$-length at most $\ell<r_n/2$.
Hence~\cref{lem:commuting-supports-EL} implies that every conjugated factor of type~2 commutes with every conjugated factor of types~3 and~4.
Conjugated factors of types~3 and~4 are supported on disjoint sheets and therefore commute as well.
We may therefore collect the factors according to their types while preserving the order within each type, which gives the claimed normal form.
\end{proof}

\begin{theorem}\label{thm:growth-estimate-EL}
For every $\alpha \in(\alpha_0,1)$ and every finitely generated associative unital ring $R$ there is a constant $C_R>0$ satisfying
\[
\gamma_{\EL_{Y_n}(R)}^{\mathcal E_n(R)}(\ell)
\leq
\exp\bigl(C_R\ell^\alpha\bigr)
\]
for every $n\geq2$ and every $0\leq\ell\leq r_n/3$.
\end{theorem}
\begin{proof}
Fix $\alpha\in(\alpha_0,1)$.
For $i\in\{1,2,3,4\}$, let $P_i^R(\ell)$ denote the set of all possible factors $p_i$ appearing in the normal form from~\cref{lem:normal-form-EL} for words of length at most $\ell\leq r_n/3$ over $\mathcal E_n(R)\cup\mathcal E_n(R)^{-1}$.
Since $\mathcal G_n$ is a quotient of $\mathcal G$, it follows from \cref{thm:G-subexponential} that
\[
|P_1^R(\ell)|
\leq
\exp\bigl(C_1\ell^{\alpha_0}\bigr)
\]
for a constant $C_1>0$ that is independent of $n$ and $R$.
Let $C_0>0$ be a constant such that the estimates from~\cref{prop:BE-finite} hold with $C_0$ for both $\rho_n$ and $\eta_n$.
In what follows we use the notation of~\cref{lem:normal-form-EL} and set $u\defeq g_1 \ldots g_\ell$.
Each conjugate $(a_i^{(2)})^{g_i\cdots g_\ell}$ is supported on the fiber over the point $\eta_ng_i\cdots g_\ell$.
The factor $p_2$ is therefore supported on the fibers over the points of the inverted orbit $\OO_{\eta_n}(u)$.
By~\cref{prop:BE-finite}, there are at most $\exp(C_0\ell^{\alpha_0})$ inverted orbits $\OO_{\eta_n}(u)$ arising from words $u$ of length at most $\ell$, and each of them has cardinality at most $C_0\ell^{\alpha_0}$.
At every point of such an orbit, the contribution of $p_2$ lies in a conjugate of $\mathcal M_{\eta_n}(Y_n;R)$, whose order is at most $2^3\cdot4!$ by~\cref{lem:elementary-monomial}.
Thus we have
\[
|P_2^R(\ell)|
\leq
\exp(C_0\ell^{\alpha_0}) \cdot
(2^3\cdot4!)^{C_0\ell^{\alpha_0}}
\leq
\exp(C_2\ell^{\alpha_0})
\]
for every $\ell\leq r_n/3$ and a suitable constant $C_2>0$ that is independent of $n$ and $R$.
The factor $p_3$ is a signed monomial matrix supported on the first sheet, and its support is contained in
\[
\bigl\{\rho_ng_i\cdots g_\ell,\ \theta_ng_i\cdots g_\ell
\mid 1\leq i\leq\ell\bigr\}\times\{1\}.
\]
As in the proof of~\cref{lem:count-p3}, this set is contained in $\OO_{\rho_n}(\widetilde u)\times\{1\}$, where
\[
\widetilde u
\defeq
a_ng_1a_n\,a_ng_2a_n\cdots
a_ng_{\ell-1}a_n\,a_ng_\ell
\]
is a word of length at most $3\ell-1\leq r_n$ over $S_n\cup\{1\}$.
By~\cref{prop:BE-finite}, there are at most $\exp\bigl(C_0(3\ell)^{\alpha_0}\bigr)$ inverted orbits $\OO_{\rho_n}(\widetilde u)$ arising in this way, and each of them has cardinality at most $m_\ell\defeq\bigl\lceil C_0(3\ell)^{\alpha_0}\bigr\rceil$.
For a fixed such orbit $O$, the number of signed monomial matrices on the basis elements indexed by $O\times\{1\}$ is at most $2^{m_\ell}m_\ell!$.
We therefore obtain
\[
|P_3^R(\ell)|
\leq
\exp\bigl(C_0(3\ell)^{\alpha_0}\bigr)
2^{m_\ell}m_\ell!
\leq
\exp\bigl(C_3\ell^{\alpha_0}\log\ell\bigr)
\]
for every $2\leq\ell\leq r_n/3$ and a suitable constant $C_3>0$ that is independent of $n$ and $R$.
To estimate the cardinality of $P_4^R(\ell)$, we note that each of the conjugates $(a_i^{(4)})^{g_i\cdots g_\ell}$ is supported on the second and third sheet over $\rho_ng_i\cdots g_\ell$.
The factor $p_4$ is therefore supported on the fibers over the points of the inverted orbit $\OO_{\rho_n}(u)$.

At a fixed point $x$ of this orbit, the product of the contributions of the conjugates is, by~\eqref{eq:elementary-addition}, a matrix of the form $e_{(x,2),(x,3)}(c_x)$, where $c_x$ is a sum of at most $\ell$ elements of $A$.
Writing $c_x=\sum_{a\in A}m_{x,a}a$ with integers $|m_{x,a}|\leq\ell$, we see that there are at most $(2\ell+1)^{|A|}$ possible values of $c_x$.
Since, by~\cref{prop:BE-finite}, there are at most $\exp(C_0\ell^{\alpha_0})$ inverted orbits $\OO_{\rho_n}(u)$, each of cardinality at most $C_0\ell^{\alpha_0}$, we obtain
\[
|P_4^R(\ell)|
\leq
\exp(C_0\ell^{\alpha_0})
(2\ell+1)^{|A|\,C_0\ell^{\alpha_0}}
\leq
\exp\bigl(C_4(R)\ell^{\alpha_0}\log\ell\bigr)
\]
for every $2\leq\ell\leq r_n/3$ and a suitable constant $C_4(R)>0$ that is independent of $n$.
Since
\[
\ell^{\alpha_0} \log(\ell)
\leq C_{\alpha}\ell^{\alpha}
\]
for every $\ell \geq 2$ and a suitable constant $C_{\alpha}>0$, the preceding estimates yield a constant $C_R>0$ with
\[
\gamma_{\EL_{Y_n}(R)}^{\mathcal E_n(R)}(\ell)
\leq
\prod_{i=1}^4|P_i^R(\ell)|
\leq
\exp\bigl(C_R\ell^{\alpha}\bigr)
\]
for every $n\geq2$ and every $2\leq\ell\leq r_n/3$.
\end{proof}

We are now ready to prove the estimates for the infimal growth rates of elementary linear groups stated in the introduction.

\begin{proof}[Proof of~\cref{thm:el}]
To prove the theorem we fix $\beta\in(0,1-\alpha_0)$ and $\alpha\in(\alpha_0,1-\beta)$.
Let $C_R>0$ be the constant provided by~\cref{thm:growth-estimate-EL} for this choice of $\alpha$.
For every $n$ we write
\[
\ell_n\defeq\left\lfloor\frac{r_n}{3}\right\rfloor.
\]
For all sufficiently large $n$, the estimate from~\cref{thm:growth-estimate-EL} gives us
\begin{align*}
\omega\bigl(\EL_{2^{n+2}}(R)\bigr)
&\leq
\omega\bigl(\EL_{Y_n}(R),\mathcal E_n(R)\bigr)\\
&\leq
\gamma_{\EL_{Y_n}(R)}^{\mathcal E_n(R)}(\ell_n)^{1/\ell_n}\\
&\leq
\exp\bigl(C_R\ell_n^{\alpha-1}\bigr).
\end{align*}
Since $\ell_n\geq2^{n-2}$ for all sufficiently large $n$ we have
\[
C_R\ell_n^{\alpha-1}
\leq
C_R\,2^{-(1-\alpha)(n-2)}.
\]
As $\beta<1-\alpha$ by the choice of $\alpha$, the right-hand side is at most $2^{-\beta(n+2)}$ for all sufficiently large $n$.
Consequently,
\[
\omega\bigl(\EL_{2^{n+2}}(R)\bigr)
\leq
\exp\bigl(2^{-\beta(n+2)}\bigr)
\]
for all sufficiently large $n$, which completes the proof of~\cref{thm:el}.
\end{proof}

\clearpage
\appendix
\section{Notation}
\label{app:notation}

The tables collect notation used repeatedly in the paper. \medskip

\begingroup
\small
\setlength{\tabcolsep}{4pt}
\renewcommand{\arraystretch}{1.18}

\begin{center}\textbf{Growth and inverted orbits}\end{center}\smallskip

\begin{tabular}{@{}p{0.2\textwidth}p{0.5\textwidth}p{0.3\textwidth}@{}}
\textbf{Symbol} & \textbf{Meaning} & \textbf{Where} \\\hline
$\gamma_G^S(n)$ & Number of elements in the word ball of radius $n$ for the finite generating set $S$. & Section~1 \\[3pt]
$\alpha_0$, $\eta$ & $\alpha_0=\log(2)/\log(2/\eta)$, where $\eta$ is the real root of $x^3+x^2+x-2$. & \cref{not:alpha0} \\[3pt]
$\OO_\omega(w)$, $\delta_\omega(w)$ & Inverted orbit of the word $w$ at $\omega$, and its cardinality. & \cref{def:inverted-orbit} \\[3pt]
$\Delta_\omega^A(n)$ & Largest inverted orbit cardinality for words over $A$ of length at most $n$. & \cref{def:inverted-orbit} \\[3pt]
$N_\omega^A(n)$ & Number of distinct inverted orbits of words over $A$ of length at most $n$. & \cref{def:inverted-orbit} \\[3pt]
\end{tabular}
\medskip

\begin{center}\textbf{The binary tree and Schreier graphs}\end{center}\smallskip

\begin{tabular}{@{}p{0.2\textwidth}p{0.5\textwidth}p{0.3\textwidth}@{}}
\textbf{Symbol} & \textbf{Meaning} & \textbf{Where} \\\hline
$\T$, $L_n$ & Binary rooted tree; its $n$-th level $L_n=\{0,1\}^n$. & \cref{sec:Grigorchuk-background} \\[3pt]
$\mathcal G$, $S$ & First Grigorchuk group and its generating set $S=\{a,b,c,d\}$. & \cref{sec:Grigorchuk-background} \\[3pt]
$\mathcal G_n$, $S_n$ & Image of $\mathcal G$ on $L_n$ and the generators $\{a_n,b_n,c_n,d_n\}$. & \eqref{eq:grigorchuk-quotient}; \cref{sec:inverted-orbits-cantor} \\[3pt]
$\Gamma_n$, $\Gamma_\infty$ & Schreier graphs of the action on $L_n$ and on the orbit of $\rho=1^\infty$, respectively. & \cref{sec:Grigorchuk-background} \\[3pt]
$\rho_n$, $\eta_n$, $\rho$ & Distinguished vertices $1^n$, $1^{n-1}0$, and boundary point $1^\infty$. & \eqref{eq:distinguished-vertices}; \cref{sec:Grigorchuk-background}  \\[3pt]
$\theta_n$ & Vertex $\rho_na_n=0\,1^{n-1}$. & \cref{sec:inverted-orbits-cantor} \\[3pt]
\end{tabular}
\medskip

\begin{center}\textbf{Cantor spaces and Thompson groups}\end{center}\smallskip

\begin{tabular}{@{}p{0.2\textwidth}p{0.5\textwidth}p{0.3\textwidth}@{}}
\textbf{Symbol} & \textbf{Meaning} & \textbf{Where} \\\hline
$\mathfrak C$, $w\mathfrak C$ & Cantor set $\{0,1\}^{\N}$ and the cylinder with finite prefix $w$. & \cref{sec:V-background}; \eqref{eq:cylinder} \\[3pt]
$V_Z$ & Higman--Thompson group on the copies indexed by $Z$. & \cref{sec:V-background} \\[3pt]
$\mathfrak C_Z$, $\mathfrak C_z$ & Disjoint union of Cantor copies indexed by $Z$, and its $z$-th copy. & \eqref{eq:cantor-copies}; \cref{sec:V-background}  \\[3pt]
$w\gamma$, $wH$ & Copy of $\gamma$, or of the subgroup $H\leq V$, acting inside $w\mathfrak C$. & \cref{not:local-V} \\[3pt]
$\RiSt_G(U)$ & Subgroup of $G$ fixing the complement of $U$ pointwise. & \cref{not:rigid-stabilizer} \\[3pt]
$X_i$ &  Standard generators of Thompson's group $F\leq V$, with $X_i=1^iX_0$ for $i\geq1$. & \eqref{eq:thompson-X0}; \cref{sec:V-background}  \\[3pt]
$Y_n$, $\pr_n$ & $Y_n=L_n\times\{1,2,3,4\}$ and the projection $Y_n\to L_n$. & \cref{sec:inverted-orbits-cantor}
\end{tabular}

\bibliographystyle{alpha}
\bibliography{literature.bib}
\end{document}